\documentclass{amsart}
\usepackage{amssymb}
\usepackage{amsmath}
\usepackage{graphicx}
\usepackage{amsfonts}
\newtheorem{theorem}{Theorem}
\theoremstyle{plain}

\newtheorem{definition}{Definition}
\newtheorem{example}{Example}

\newtheorem{lemma}{Lemma}

\newtheorem{proposition}{Proposition}
\newtheorem{remark}{Remark}

\numberwithin{equation}{section}

\author{Samson Saneblidze}

\title{The bitwisted Cartesian   model for the free loop fibration}
\date{}

\dedicatory{To Murray  Gerstenhaber  and  Jim Stasheff}


\address{A. Razmadze Mathematical
Institute\\
Department of Geometry and Topology\\
M. Alexidze st., 1 \\
 0193, Tbilisi,  Georgia}
\email{SANE@@rmi.acnet.ge}

\begin{document}

\begin{abstract}
 Using the notion of
truncating twisting function from a simplicial set to a cubical set a special,
bitwisted, Cartesian product of these sets is defined.  For the universal truncating
twisting function, the (co)chain complex of the corresponding bitwisted Cartesian
product agrees with the standard Cartier (Hochschild) chain complex of the simplicial
(co)chains. The modelling polytopes $F_n$ are constructed.  An explicit diagonal on
$F_n$ is defined and a multiplicative model for the free loop fibration $\Omega
Y\rightarrow \Lambda Y\rightarrow Y$ is obtained. As an application we establish an
algebra isomorphism $H^*(\Lambda Y;\mathbb{Z}) \approx S(U)\otimes \Lambda(
s^{_{-\!1}}\!U)$ for the polynomial cohomology algebra $H^*(Y;\mathbb{Z})=S(U).$

\end{abstract}
\maketitle

\section{Introduction}

Let $\Lambda Y$ denote the free loop space of a topological space $Y,$ i.e. the space
of  all continuous maps from the circle $S^1$ into $Y,$ and let $\Omega Y \rightarrow
\Lambda Y\overset{\xi}{\rightarrow} Y$ be the free loop fibration. Since $\xi$ can be
viewed as obtained from the path fibration $\Omega Y \rightarrow P
Y\overset{}{\rightarrow} Y$ by means of the conjugation action $ad:\Omega Y \times
\Omega Y\rightarrow \Omega Y $ \cite{McCleary}, one could apply  \cite{KScubi} to
construct for $\xi$ the twisted tensor product $C_{\ast}(Y)\otimes_{\tau_{\ast}} \Omega
C^{\ast}(Y).$ However, the induced action $ad_{\ast}:\Omega C_{\ast}(Y)\otimes \Omega
C_{\ast}(Y) \rightarrow \Omega C_{\ast}(Y) $ is hardly to write down by explicit
formulas; an alternative way for modelling $\Lambda Y$ is the Cartier chain complex
$\Lambda C_{\ast}(Y)$ of the singular simplicial chain coalgebra $C_{\ast}(Y)$ thought
of as a specific, \emph{bitwisted},  tensor product $\Lambda
C_{\ast}(Y)=C_{\ast}(Y)\,\,_{\tau_{\ast}}\!\!\!\otimes_{\tau_{\ast}} \Omega
C_{\ast}(Y)$ \cite{Cartier}, \cite{McCleary}. Accordingly, for a twisting truncating
function $\tau\! : \!X \rightarrow Q$ from a 1-reduced simplicial set to a monoidal
cubical set $Q,$
 we modify the twisted
Cartesian product $X \times _{\tau}Q $ from \cite{KScubi} to obtain a \emph{bitwisted
Cartesian product} $X\,_{\tau}\!\!\times_{\tau} Q $
 such that $\Lambda C_{\ast}(X)=
C^{\odot}_{\ast}(X \,_{\tau}\!\!\times_{\tau} Q )$ whenever $Q=\mathbf{\Omega} X,$ a
monoidal cubical set,  and $\tau=\tau_{_{U}} \! : \! X\rightarrow \mathbf{\Omega} X,$
the universal truncating twisting function, constructed in \cite{KScubi}.
 Dually, for the Hochschild complex
$\Lambda C^{\ast}(X)$ of the simplicial cochain algebra $ C^{\ast}(X)$ we get the
inclusion of cochain complexes $\Lambda C^{\ast}(X)\subset C^{\ast}_{\odot}(X
\,_{\tau}\!\!\times_{\tau} \mathbf{\Omega} X)$ (here we have equality when the graded
sets have finite type). The required model for $\Lambda Y$ is obtained by taking
$X=\operatorname{Sing}^1 Y,$ the Eilenberg 1-subcomplex of the singular simplicial set
$\operatorname{Sing} Y.$

We  construct polytopes $F_n$ referred to as \emph{freehedra} and introduce   the
notion of an $F_n$-\emph{set}. The  motivation is that the bitwisted Cartesian product
$X\,_{\tau}\!\!\times_{\tau} \mathbf{\Omega} X$ above admits  such a
 combinatorial structure in a canonical way. As the standard simplices $\Delta
^n,\ n\geq 0,$ are the modelling polytopes for a simplicial set,  the Cartesian
products $F_i\times I^j,\ i,j\geq 0,$ with $I^j$ the standard $j$-cube  serve as
modelling polytopes for an $F_n$-set. A universal example for an $F_n$-set is  the
\emph{singular}  complex $\operatorname{Sing}^F Y,$
  obtained as the set of continuous maps    $I^k\times F_m\times I^{\ell}
  \rightarrow Y$ for $ k,\ell,m\geq 0.$
  The normalized (co)chain complex $C_{\ast}^{\odot}(Y)\,(C^{\ast}_{\odot}(Y))$ of
   $\operatorname{Sing} ^F Y$
   also gives
 the singular (co)homology
of $Y.$ In particular, we  construct a map of
  $F_n$-sets
\begin{equation*}
   \Upsilon: \operatorname{Sing}^{1} Y\,_{\tau}\!\!\times_{\tau} {\bf \Omega}
\operatorname{Sing}^{1} Y\rightarrow \operatorname{Sing}^{F}\Lambda Y
\end{equation*}
 that extends the
cubical map $\omega: \mathbf{\Omega}  \operatorname{Sing}^{1}Y\rightarrow
\operatorname{Sing}^I \Omega Y$ realizing Adams' cobar equivalence $\omega_*:\Omega
C_*(Y)\rightarrow C^{\Box}_*(\Omega Y)$ \cite{Adams}:
 The image
$\Upsilon(\sigma_m,\sigma'_n)$ for $m,n>1$ consists of singular $(F_m\times
I^n)$-polytopes that are purely determined by a choice of  measuring homotopy for the
\emph{commutativity} of the standard loop product of $\Omega Y$ by the inclusion
$\Omega Y\subset \Lambda Y.$ Such a relationship is possible since $F_m$ admits a
representation as an explicit subdivision of the cube $I^m=I^{m-1}\times I.$

We construct an explicit diagonal $\Delta_F$
 for  $F_n.$
In a standard way  this diagonal  together with the Serre diagonal of the cubes  yields
the diagonal of an arbitrary $F_n$-set. Consequently, the $F_n$-set structure of the
bitwisted Cartesian product $\operatorname{Sing}^{1} Y\,_{\tau}\!\!\times_{\tau} {\bf
\Omega} \operatorname{Sing}^{1} Y, $ as a by-product, determines a comultiplication on
the Cartier complex $\Lambda C_{\ast}(Y).$ Dually, we obtain a multiplication on the
Hochschild complex $\Lambda C^{\ast}(Y).$ More precisely, a combinatorial analysis of
$\Delta_F$ shows that
 the multiplication on the
Hochschild complex $\Lambda C^{\ast}(Y)$ involves
 a canonical homotopy $G$-algebra (hga) structure
on the simplicial cochain algebra   $C^{\ast}(Y)$  \cite{Baues3}, \cite{Gerst-Voron},
\cite{Getz-Jones} (compare \cite{KScubi}).
 Thus, applying the cohomology functor for
$\Upsilon$ we obtain a natural algebra isomorphism
\begin{equation*}
\Upsilon^{\ast}: HH_{\ast}(C^{\ast}(Y;\Bbbk))\overset{\approx}{\longleftarrow}
H^{\ast}(\Lambda Y;\Bbbk)
\end{equation*}
 in which on the left-hand side the product is given by formula
(\ref{formula}) and $\Bbbk$ is a commutative ring with identity. Note also that
 the multiplication on the Hochschild complex $\Lambda C^{\ast}(Y)$ is not
associative, but it can be extended to an $A_{\infty}$-algebra structure
\cite{Stasheff}.

    As an application of the above algebra isomorphism we establish the fact that the standard shuffle product on
   the Hochschild complex
     $\Lambda H$ with $H=H^*(Y;\Bbbk)$ polynomial is \emph{geometric}.
More precisely,
 given  a cocycle $z\in C^{n}(Y;\Bbbk),$ the product
 $z\smile_1 z\in C^{2n-1}(Y;\Bbbk)$ is a cocycle for  $n$  even or
 for $n$  odd too, when $z^2$ is of the second order. In any case, the
 class $[z\smile_1 z]\in H^{2n-1}(Y;\Bbbk)$ is defined and is denoted
 by $Sq_1[z].$
 Clearly, if $H=H^*(Y;\Bbbk)$ is a polynomial algebra and
$\Bbbk$ has no 2-torsion, then
 $H$ is evenly graded and
$Sq_1$ is identically zero on $H.$ However,
 $H$
may have odd dimensional generators
 for $\Bbbk=\mathbb{Z}/2 \mathbb{Z}.$
We have the following theorem that generalizes a well-known result when $\Bbbk$ is a
field.
\begin{theorem}\label{diff.forms}
Let $H=H^*(Y;\Bbbk)=S(U)$  be a polynomial algebra  such that $Sq_1=0$ on $H.$
 Let $\Lambda( s^{_{-\!1}}\! U)$
be the exterior algebra over the desuspension of the polynomial generators $U.$ Then
there are algebra isomorphisms
$$ H(\Lambda Y;\Bbbk)
\approx S(U)\otimes \Lambda( s^{_{-\!1}}\! U)\approx H(Y;\Bbbk)\otimes H(\Omega
Y;\Bbbk).
$$

\end{theorem}

Note that in an ungraded setting of $H=H(Y;\Bbbk)$  the middle term above is
interpreted as the module of differential forms denoted by $\Omega_{_{H|\Bbbk}}$
 \cite{Loday1}.
Using a \emph{strong homotopy commutative} structure of $C^{\ast}(Y;\Bbbk)$ the
multiplication on $\Lambda C^{\ast}(Y;\Bbbk)$ is constructed in \cite{Ndombol-Thomas}
for $\Bbbk$ to be a field (see also \cite{Menichi} for references). Though there is a
close relationship between the strong homotopy commutative and hga structures,  the
explicit formula for the product  on the Hochschild chain complex $\Lambda C^{\ast}(Y)$
in terms of the hga operations
 is
heavily   used here (see Lemma \ref{ederivation}). Moreover, the proof of the above
theorem suggests further calculations  of the Hochschild homology for a more wide class
of spaces having, for example, the cohomology isomorphic to smooth algebras.

Restricted to finitely generated polynomial algebras, i.e., $U$ is finite dimensional,
Theorem 1 agrees with the solution of the Steenrod problem given in \cite{Ande-Grod}.
At the end of the paper we include  an example showing that there is a tensor product
algebra $C$ but with perturbed differential such that the $E_{\infty}$-term of the
spectral sequence of $C$ is isomorphic as algebras with the $E_{\infty}$-term of the
Serre spectral sequence of the free loop fibration with the base
$Y=\mathbb{C}P^{\infty}\times \mathbb{C}P^{\infty},$
   nevertheless $H^*(C)$ differs from $H^*(\Lambda Y)$ as algebras.

Finally,  note that a large part of the paper consists of an editing excerpt from the
earlier preprint \cite{sane2000} where in particular the polytopes $F_n$ were
constructed. In the meantime 3-dimensional polytope $F_3$ appeared in \cite{Breen}(see
also \cite{Loday2}).

I would like to thank  J.-C. Thomas, J.-L. Loday and T. Pirashvili for helpful
conversations about topological and algebraic aspects of the question. Special thanks
are to Jim Stasheff for his continuous interest and suggestions. I also thank the
referee for helpful comments.

\section{some preliminaries and conventions}

We adopt the notions and notations from \cite{KScubi}.
\subsection{Cobar and Bar constructions}

Let $\Bbbk$ be a commutative ring with identity. For a $\Bbbk$-module $M,$ let $T(M)$
be the tensor algebra of $M$, i.e. $ T(M)=\oplus _{i=0}^{\infty }M^{\otimes i}$. An
element $a_{1}\otimes ...\otimes a_{n}\in M^{\otimes n}$ is denoted by $[a_{1}|\cdots
|a_{n}]$. We denote by $s^{-1}M$ the desuspension of $M$, i.e. $(s^{-1}M)_{i}=M_{i+1}$.

Let $(C,d_{C},\Delta)$ be a 1-reduced dgc, i.e. $C_0=\Bbbk,\,C_1=0.$ Denote
$\bar{C}=s^{-1}(C_{>0})$. Let $\Delta=Id\otimes1+1\otimes Id+{\Delta^{\prime}}$. The
(reduced) cobar construction $\Omega C$ on $C$ is the tensor algebra $T(\bar{C})$, with
differential $d=d_{1}+d_{2}$ defined for $\bar{c}\in\bar{C}_{>0}$ by
\begin{equation*}
d_{1}[\,\bar{c}\,]=-\left[\,\overline{d_{C}(c)}\,\right]\ \ \text{and}\ \
d_{2}[\bar{c}]=\sum(-1)^{|c^{\prime}|}\left[\,\bar{c^{\prime}}|\bar{c^{\prime\prime}
}\,\right],\ \ \ \text{for}\ \ \ {\Delta^{\prime}}(c)=\sum c^{\prime}\otimes
c^{\prime\prime},
\end{equation*}
extended as a derivation. The acyclic cobar construction $\Omega(C;C)$ is the twisted
tensor product $C\otimes\Omega C$ in which the tensor differential is twisted by the
universal twisting cochain $\tau_*: C\rightarrow\Omega C$ being an inclusion of degree
$-1.$

Let $(A,d_{A}, \mu)$ be a 1-reduced dga. The (reduced) bar construction $BA$ on $A$ is
the tensor coalgebra $T(\bar A),\ \bar A= s^{-1}(A_{>0}),$ with differential $d=d_{1}
+d_{2} $ given for $[\bar a_{1}|\dotsb|\bar a_{n}] \in T^{n}(\bar A)$ by
\begin{equation*}
d_{1}[\bar a_{1}|\dotsb|\bar a_{n}]=-\sum_{i=1}^{n} (-1)^{\epsilon_{i-1}}[\bar
a_{1}|\dotsb|\overline{d_{A}(a_{i})}|\dotsb|\bar a_{n}],
\end{equation*}
and
\begin{equation*}
d_{2} [\bar a_{1}|\dotsb|\bar a_{n}]=- \sum_{i=1}^{n-1} (-1)^{\epsilon_{i}}[\bar
a_{1}|\dotsb|\overline{a_{i}a_{i+1}}\,|\dotsb|\bar a_{n}],
\end{equation*}
where $\epsilon_{i}=\epsilon^a_i=|a_1|+\cdots +|a_i|+i.$ The acyclic bar construction
$B ( A; A)$ is the twisted tensor product $A\otimes BA$ in which the tensor
differential is twisted by the universal twisting cochain $\tau^*: BA\to A$ being a
projection of degree $1.$

\subsection{Cartier and Hochschild chain complexes}
Let $(C,d_{C},\Delta)$ be a 1-reduced dgc and let $\Delta=Id\otimes 1+\Delta_1
=1\otimes Id+\Delta_2.$ The (normalized) \emph{Cartier complex } $\Lambda C$
 of $C$ \cite{Cartier}
is  $C\otimes  {\Omega}C$ with  differential $d$ defined by $d= d_C\otimes 1 + 1
\otimes d_{ \Omega C}+ \theta _1+\theta _2,$
 where
\begin{equation*}
\begin{array}{llll}
\theta _1 (v\otimes [\bar c_1|\dotsb |\bar c_n]) = -\sum (-1)^{|v_1'|}\,v_1'\otimes
 [\bar v_1''|\bar c_1|\!\dotsb\! |\bar c_n],\newline $\vspace{1mm}$
 \\
\theta _2 (v\otimes [\bar c_1|\dotsb |\bar c_n]) = \sum (-1)^{(| v_2'|+1) (|v_2''|+
\epsilon^c_{n})}\, v_2''\otimes [\bar c_1|\!\dotsb \!|\bar c_{n}|\bar v'_2],
\newline $\vspace{1mm}$\\
\hspace{3.1in} \Delta_i(v)=\sum v'_i\otimes v''_i,\,i=1,2.

\end{array}
\end{equation*}

The homology of  $\Lambda C$ is called the Cartier homology
 of a dgc
$C$ and is denoted by $HH_*(C).$

Note that the components $\theta_1$ and $\theta_2$ above can be thought of as obtained
by applying the universal twisting cochain $\tau_{*}: C\rightarrow \Omega C$ on the
tensor product $C\otimes \Omega C$  twice:
$$C\otimes \Omega C \overset{\Delta\otimes 1}\longrightarrow
C\otimes C\otimes \Omega C \overset{{1\otimes \tau_*\otimes 1+(\!1\otimes
1\otimes\tau_*\!)T}} {\longrightarrow}C\otimes \Omega C\otimes \Omega
C\overset{1\otimes \mu}{\longrightarrow} C\otimes \Omega C,
$$
where $T:C\otimes (C\otimes \Omega C)\rightarrow  (C\otimes \Omega C)\otimes C;$
consequently, $\Lambda C$ is a \emph{bitwisted tensor product} $C\,_{\tau_*}\!\!\otimes
_{\tau_*}\Omega C.$

The (normalized) \emph{Hochschild complex}  $\Lambda A$
 of a 1-reduced  associative   dg   algebra $A$ (\cite{Loday1})
is  $A\otimes  {B}A$ with  differential $d$ defined by $d= d_A\otimes 1 + 1 \otimes d_{
B A}+ \theta ^1+\theta ^2 ,$ where
\begin{equation*}
\begin{array}{lll}
\theta ^1(u\otimes [\bar a_1|\dotsb |\bar a_n]) = -(-1)^{|u|}ua_1\otimes [\bar
a_2|\dotsb |\bar a_n],\newline $\vspace{1mm}$ \\

\theta ^2 (u\otimes [\bar a_1|\dotsb |\bar a_n]) =(-1)^{(| a_n|+1)
(|u|+\epsilon^a_{n-1})} a_n u\otimes [\bar a_1|\dotsb |\bar a_{n-1}].
\end{array}
 \end{equation*}

The homology of  $\Lambda A$ is called the Hochschild homology of a dga $A$ and is
denoted by $HH_*(A).$

Dually, the components $\theta^1$ and $\theta^2$  can be thought of as obtained by
applying the universal twisting cochain $\tau^{*}:B A\rightarrow A$ on the tensor
product $A\otimes BA$  twice; consequently, $\Lambda A$ is a \emph{bitwisted tensor
product} $A\,_{\tau^*}\!\!\otimes _{\tau^*}BA.$

\section{The  polytopes $F_n$}

It is well known that the standard cube $I^n$ can be viewed as obtained from the
standard simplex $\Delta^n$  with vertices $(v_0,v_1,...,v_n)$ by a truncating
procedure starting either at the minimal vertex $v_0$ or at the maximal vertex $v_n;$
One gets a polytope, called a freehedron and denoted by $F_n,$ when  the truncations
start at both vertices $v_0$ and $v_n$ simultaneously; since these two truncations do
not meet, we can begin the truncation procedure  at vertex $v_0$ to obtain first the
standard cube $I^n,$ and then continue the same procedure at vertex $v_n$ to recover
$F_n.$ Thus, we obtain the canonical cellular projection, a "healing" map, $\varphi:
F_n\rightarrow \Delta^n$ such that it factors through the projections $\phi:
F_n\rightarrow I^n$ and $\psi:I^n\rightarrow \Delta^n$ (see Figures 1-3).

It is convenient, to regard $F_n$ as a subdivision of $I^n$ and to give the following
inductive construction of $F_n$ whose faces are labelled three types of face operators
$d^0_i,d^1_i$ and $d^2_i$  with $d^1_1=d^2_1.$ Let $F_0$ and $F_1$ be a point and an
interval respectively. If $F_{n-1}$ has been constructed, let $e^{\epsilon}_{i}$ denote
the face $(x_1,...,x_i,\epsilon,x_{i+1},....,x_{n-1})\subset I^n$ where $\epsilon=0,1$
and $1\leq i\leq n.$  Then $F_n$ is the subdivision of $F_{n-1}\times I$ given below
and its various $(n-1)$-faces are labelled as indicated:
\begin{equation*}
\begin{tabular}{c|cc}
$\underset{\ }{\text{\textbf{Face of }}F_{n}}$ & \textbf{Face operator} &
\\ \hline

 $e^{0}_{i}$ & $d^{0}_{i} ,$ & \multicolumn{1}{l}{$1\leq
i\leq n$} \\

$e^{1}_{i}$ & $d^{1}_{i},$ & \multicolumn{1}{l}{$2\leq i\leq n$} \\

$d^{2}_{i}\times I$ & $d^{2}_{i},$ &
\multicolumn{1}{l}{$1\leq i\leq n-2$} \\

$d^{2}_{n-1}\times\left[0,\frac{1}{2}\right]^{\strut } $ & $d^{2}_{n-1},$
& \multicolumn{1}{l}{} \\

$d^{2}_{n}\times\left[\frac{1}{2},1  \right]^{\strut } $ & $d^{2}_{n},$ &
\multicolumn{1}{l}{}
\end{tabular}
\end{equation*}
\vspace{0.2in}

\unitlength=1.00mm \special{em:linewidth 0.4pt} \linethickness{0.4pt}
\begin{picture}(112.33,43.33)
\put(9.67,40.00){\line(0,-1){20.00}} \put(9.67,20.00){\line(1,0){20.33}}
\put(30.00,20.00){\line(0,1){20.00}} \put(30.00,40.00){\line(-1,0){20.33}}
\put(9.67,40.00){\circle*{1.33}} \put(30.00,40.00){\circle*{1.33}}
\put(9.67,20.00){\circle*{1.33}} \put(30.00,20.00){\circle*{1.33}}
\put(30.00,30.00){\circle*{1.33}} \put(19.67,16.67){\makebox(0,0)[cc]{$d^0_2$}}
\put(19.67,43.33){\makebox(0,0)[cc]{$d^1_2$}}
\put(6.00,29.67){\makebox(0,0)[cc]{$d^0_1$}}
\put(33.33,35.00){\makebox(0,0)[cc]{$d^2_2$}}
\put(33.00,24.67){\makebox(0,0)[cc]{$d^2_1$}}
\put(27.50,24.67){\makebox(0,0)[cc]{$d^1_1$}} \put(94.00,40.00){\line(2,-1){17.67}}
\put(111.67,31.00){\line(-1,0){22.00}} \put(89.67,31.00){\line(0,-1){20.33}}
\put(89.67,10.67){\line(1,0){22.00}} \put(111.67,10.67){\line(0,1){20.33}}
\put(89.67,10.67){\line(-2,1){17.67}} \put(72.00,19.67){\line(0,1){20.33}}
\put(72.00,19.67){\line(1,0){17.00}} \put(90.34,19.67){\line(1,0){3.67}}
\put(94.34,40.00){\line(0,-1){8.33}} \put(94.34,30.33){\line(0,-1){10.67}}
\put(94.34,19.67){\line(2,-1){17.33}} \put(102.67,35.33){\line(0,-1){4.00}}
\put(102.67,30.33){\line(0,-1){15.00}} \put(94.34,28.33){\line(5,-2){8.33}}
\put(72.00,40.00){\line(1,0){22.00}} \put(89.67,31.00){\line(-2,1){17.67}}
\put(94.34,40.00){\circle*{1.33}} \put(72.00,40.00){\circle*{1.33}}
\put(94.34,28.33){\circle*{1.33}} \put(102.67,25.00){\circle*{1.33}}
\put(102.67,35.67){\circle*{1.33}} \put(111.67,31.00){\circle*{1.33}}
\put(89.67,31.00){\circle*{1.33}} \put(94.34,19.67){\circle*{1.33}}
\put(102.67,15.33){\circle*{1.33}} \put(111.67,10.67){\circle*{1.33}}
\put(89.67,10.67){\circle*{1.33}} \put(72.00,19.67){\circle*{1.33}}
\put(98.00,33.67){\makebox(0,0)[cc]{$d^2_3$}}
\put(98.67,22.33){\makebox(0,0)[cc]{$d^2_2$}} \ \
\put(107.60,23.67){\makebox(0,0)[cc]{$_{d^2_1(d^1_1)}$}}
\put(90.67,35.00){\makebox(0,0)[cc]{$d^1_3$}}
\put(83.34,30.33){\makebox(0,0)[cc]{$d^1_2$}}

\put(83.34,17.33){\makebox(0,0)[cc]{$d^0_1$}}
\put(92.34,16.67){\makebox(0,0)[cc]{$d^0_3$}}
\put(99.34,13.33){\makebox(0,0)[cc]{$d^0_2$}}
\end{picture}

\vspace{-0.2in}
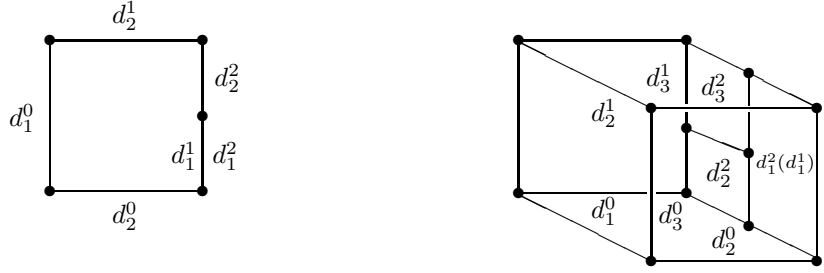
\begin{center} Figure 1: $F_{n}$ as a subdivision of $F_{n-1}\times I$
for $n=2,3.$
\end{center}

\vspace{0.2in}

 Thus, $F_2$ is a pentagon, $F_3$ has eight 2-faces (4 pentagon and 4 quadrilateral), 18
edges and 12 vertices (compare \cite{Breen}). In particular, the sequence of
codimension 1 faces of $F_n,\,n\geq1, $ is an arithmetic progression with difference 3.

\subsection{The singular $F_n$-set}
Before we give the notion of an \emph{abstract $F_n$-set} below, let consider its
universal example,
 the \emph{singular} $F_n$-set  $\operatorname{Sing}^F Y$ of a topological space
 $Y$, i.e.
the set of all continuous maps
  $\{I^{k} \times F_{m}\times
I^{\ell}\rightarrow Y \}_{
 m,k,\ell\geq 0}.$
The face  and degeneracy operators  are defined as follows.

Obviously,  the face $d^1_i(F_m)$ is homeomorphic to $F_{m-1}$ for each $i,$ while
according to the orientation of the cube $I^m$ the faces
 $d^0_i(F_m)$ and
$d^2_i(F_m)$ are homeomorphic  to $F_{i-1}\times I^{m-i}$ and $ I^{i-1}\times  F_{m-i}$
respectively; let these homeomorphisms be realized by the following inclusions
$$
\begin{array}{l}
\bar\delta ^{0} _i : F_{i-1}\times I^{m-i} \hookrightarrow F_m,\\
\bar \delta ^{1} _i :F_{m-1} \hookrightarrow F_m,\\
\bar \delta ^{2} _i: I^{i -1}\times F_{m-i}\hookrightarrow F_m \\
\end{array}
 $$
 with $\bar \delta^2_1=\bar \delta^1_1.$
Let
 $\iota^{\epsilon}_{i}:I^{k}\hookrightarrow I^{k+1}$
be the inclusion defined by $\iota^{\epsilon}_{i}(I^k)=e^{\epsilon}_{i}.$
 Given $k,\ell,m\geq 0$ with  $m+\ell=r_1,$ $k+m+\ell=r_2,$
 define the  inclusions
\[\delta^{\epsilon}_i: I^{k_{\epsilon,i}}\times F_{m_{\epsilon,i}}\times
I^{\ell_{\epsilon,i}} \hookrightarrow I^k\times F_m\times I^{\ell}\]
 by
$$
\delta^{\epsilon}_{i}=\left\{ \!\!\!
\begin{array}{llll}
 1\!\times \!\bar \delta^0_i \!\times \!1,&
\epsilon=0, & \!\! (k,m,\ell)_{0,i}=(k,m-i,\ell+i-1)  , &\! 1\leq i\leq m,\vspace{1mm}\\

1\!\times \!\bar \delta^1_i\!\times \!1,&
\epsilon=1,  &\!\!  (k,m,\ell)_{1,i}=(k,m-1,\ell)    , & \!1\leq  i\leq m,\vspace{1mm}\\

1\!\times \!1\!\times\! \iota^{\epsilon}_{_{i-m}},\! \! & \epsilon=0,\!1,  &
 \!\!
(k,m,\ell)_{\epsilon,i}=(k,m,\ell-1) , &\!
 m< i\leq r_1,\vspace{1mm}\\

\iota^{\epsilon}_{_{i-r_1}}\!\times \!1 \!\times \!1,\! \! & \epsilon=0,\!1,  &\!\!
(k,m,\ell)_{\epsilon,i}=(k-1,m,\ell)
 , &\!
 r_1     < i\leq r_2,\vspace{1mm}\\

1\!\times\! \bar \delta^2_i\!\times \!1,\!\!& \epsilon=2, &\!\!
(k,m,\ell)_{2,i}=(k+i-1,m-i,\ell), &\! 1\leq i\leq m.
\end{array}
\right.
$$
 Then define
 the face
operators
$$ d^{\epsilon}_i :  (\operatorname{Sing}^F Y)^{m,n} \rightarrow
(\operatorname{Sing}^FY)^{m_{\epsilon,i}\,,\,n_{\epsilon,i}}$$ for   $   f\in
(\operatorname{Sing} ^F Y)^{m,n}, $\, $
 f: I^{k}\times F_{m}\times I^{\ell}\rightarrow Y ,\, n=k+\ell,$ \, by
 $d^{\epsilon}_i(f) =f\circ \delta^{\epsilon}_{i},$
$\epsilon=0,1,2.$

 Given $1\leq i\leq n+1,$ define the degeneracy operators
$$
\eta_i : ( \operatorname{Sing}^F Y)^{m, n} \rightarrow (\operatorname{Sing}^F Y)^{m,
n+1}
 $$
by
 \[\eta_i(f)=\left\{\begin{array}{lll}
 f\circ (  1 \times 1\times  \bar\eta_{i}
), & 1\leq  i\leq \ell+1\\
f\circ (\bar\eta_{i-\ell}\times1 \times 1 ),& \ell+1< i\leq n+1, & n=k+\ell,
\end{array}
\right.
\]
with \ \  $ \bar\eta_i: I^{r+1}\rightarrow I^{r},\ \ \
\bar\eta_{i}(x_1,...,x_{r+1})=(x_1,...,x_{i-1},x_{i+1},...,x_{r+1}). $

Thus, we obtain the singular $F_n$-set
 $$\{\operatorname{Sing}^{F}
Y=\bigcup_{k,m,\ell\geq 0}\left[(\operatorname{Sing}^{F} Y)^{m,k+\ell}=\{I^{k}\times
F_m\times I^{\ell}\rightarrow Y\}\right], d^{0}_i,d^1_i,d^2_i,\eta_i\}$$ the face and
degeneracy operators of which satisfy the following equalities:

\begin{equation}\label{hohaxioms}
\begin{array}{ll}
d_i^0 d_j^0=d_{j-1}^0  d_i^0,   \, \ \ \  \ \ \ \ \ \ \  \ \ \ \ \ \ \ i<j,  \vspace{1mm} \\

d_i^1 d_j^1=d_{j-1}^1 d_i^1,  \, \ \ \  \ \  \ \ \ \ \ \ \ \ \ \ \ \ i<j\ \ \text{and}\
\ (i,j)\neq (1,2)\ \text{for}\
 m>0, \vspace{1mm}\\

d_i^1d_j^0= \left\{ \! \! \! \begin{array}{lll}  d_{j-1}^0d_i^1,  &\!  \, \ \ \ \ \ \ \
\ \ \ \
\ \   i<j, \vspace{1mm} \\
                            d_{j}^0d_{i+1}^1,  &\!

                          \,  \ \  \ \  \ \ \ \ \ \ \ \ \  i\geq j,
\end{array}
\right.
\vspace{1mm}\\
d^2_id^0_j=d^0_{j-i}d^2_i, \vspace{1mm} \\

d_i^2d_j^1= \left\{ \! \! \! \begin{array}{lll}   d_{j-i}^1 d_{i}^2,  &   \,   i< j-1,
\vspace{1mm} \\
            d_{m+n+j-i-2}^1d_{i+1}^2,  &  \,  i\geq j-1>0,
\end{array}
\right.
\vspace{1mm} \\

d^2_id^2_j=d^0_{m+n-i}d^2_{i +j},\vspace{1mm}\\

d_{i}^{\epsilon }\eta_{j}=\left\{  \! \! \!
\begin{array}{ccc}
\eta_{j-1}d_{i}^{\epsilon }  & \ \  i<m+j, \\
1 &                           \ \   i=m+j, \\
\eta_{j}d_{i-1}^{\epsilon } &    \ \ \ \ \ \ \ \ \ \ \ \  \ \ i>m+j,\ \epsilon=0,1,
\end{array}
 \right.
\vspace{1mm}\\

d^{2}_i \eta_j = \eta_jd^2_i,
\vspace{1mm}\\

\,\eta_i\eta_j=\eta_{j+1}\eta_i,    \ \  \ \  \ \ \ \ \  \  \ \ \ \ \ \ \  \ \  i\leq
j.
\end{array}
\end{equation}

Note that the compositions $d^1_1d^1_2$ and $d^1_1d^1_1$ that are eliminated from the
second equality above are involved  in the fifth and sixth ones as
$d^2_1d^1_2=d^2_{m+n-1}d^2_{2}$ and $d^2_1d^2_1=d^2_{m+n-1}d^2_{2}$ by taking into
account that $d^1_1=d^2_1.$

\subsection{Abstract freehedral sets} For a general theory of polyhedral sets, see
for example \cite{thesis}, \cite{DWjones}. Motivated by the underlying combinatorial
structure of the bitwisted Cartesian product $\mathbf \Lambda X$ below we give the
following
\begin{definition}
 An \underline{$F_n$-set ${\mathcal {CH}}$} is a bigraded set
$  {{\mathcal {CH}}}= \{{{\mathcal {CH}}}^{m, n}, m, n\geq 0\}$
  with total grading $m+n $ and
 three types of faces operators $d_i^0,\, d_i^1,\, d_i^2$
 with $d_1^1=d^2_1$ for $m> 0,$
$$
\begin{array}{ll}
 d^0_i : {\mathcal {CH}}^{m, n} \rightarrow  {\mathcal {CH}}^{i-1,m+n-i}, & 1\leq i\leq m,\\
 d^0_i :
{\mathcal {CH}}^{m, n} \rightarrow  {\mathcal {CH}}^{m,n-1}, & m< i\leq m+n,
\vspace{1mm} \\

 d^1_i :
{\mathcal {CH}}^{m, n} \rightarrow  {\mathcal {CH}}^{m-1,n}, & 1\leq i\leq m,
\\
 d^1_i : {\mathcal {CH}}^{m, n} \rightarrow  {\mathcal {CH}}^{m,n-1}, & m< i\leq m+n,
 \vspace{1mm}\\
 d^2_i : {\mathcal {CH}}^{m, n} \rightarrow  {\mathcal {CH}}^{m-i,n+i-1}, & 1\leq i\leq m,
 \\
\end{array}
$$
and degeneracy operators
$$
\begin{array}{ll}
\eta_i : {\mathcal {CH}}^{m, n} \rightarrow {\mathcal {CH}}^{m,n+1}, & 1\leq i\leq n+1   \\
\end{array}
 $$
satisfying  structural identities (\ref{hohaxioms}).

A  \underline{morphism} of $F_n$-sets is a family of maps $f=\{ f_{m,n} \},$ $f_{m,n}:
{\mathcal {CH}}^{m,n}
 \rightarrow {{\mathcal {CH}}'}^{m,n},$
commuting with all face and degeneracy operators.
\end{definition}

Note that unlike $d^{0}_i$ and $d^1_i$ the face operator $d^2_i$ is defined only for
$1\leq i \leq m.$ Consequently, for $m=0$ we  have  only of two types of face operators
$d^{\epsilon}_i,\epsilon=0,1,$ acting on ${\mathcal {CH}}^{0, n}$ that satisfy the
standard \emph{cubical} relations;
 so that, in every $F_n$-set
the subset $\{{\mathcal {CH}}^{0, r}\}_{r\geq 0}$ together with  the
 operators $d^0,$  $d^1$ and $\eta_i$ forms a cubical
set. Furthermore, the  operators $d^0_1$ and $d^2_{m}$  have the image in this subset.

The $F_n$-set structural relations also can be conveniently verified by means of the
following combinatorics of the polytopes $F_m$ (compare with Proposition 3.2 in
\cite{KScubi}). The top dimensional cell of $F_m$ is identified with the set
$0,1,...,m],$ while any proper $q$-face $u$ of $F_{m}$ is expressed as (see Figure 2)
\begin{multline*}
\!\!\!\!\!u\!= i_{s_t},\!...,i_{s_{t+1}}][i_{s_{t+1}},\!...,i_{s_{t+2}}]
...[i_{s_{k-1}},\!...,i_{s_{k}},\!m] [0,i_1,\!...,i_{s_1}][i_{s_1},\!...,i_{s_{2}}]
...[i_{s_{t-1}},\!...,i_{s_t}],
\\
0<i_{1}<\ldots <i_{s_{t}}<\ldots <i_{s_{k}}<m,\,\,  q=s_k-k+1.
\end{multline*}
where for $t=0$ the face $u$ is assumed to have   the form
\[u=i_{1},...,i_{s_{1}}][i_{s_{1}},...,i_{s_{2}}] ...[i_{s_{k-1}},...,i_{s_{k}},m]\]
with $0\leq i_{1}<\ldots <i_{s_{k}}<m.$
\begin{proposition}
\label{Fcubi}
 Let the face  operators $d^0,d^1,d^2$
act on  $ a_0,a_1,\!...,a_m][b_0,\!...,b_{n+1}]$ (thought of as the top cell of
$F_m\times I^n$) by
\[
\begin{array}{llll}
 a_0,a_1,\!...,a_m][b_0,\!...,b_{n+1}]\overset{d^{0}_{i}}{\rightarrow}\!
\left\{\!\!\!
 \begin{array}{lll}
a_0,a_1,\!...,a_{i-1}][a_{i-1},\!...,a_{m}][b_0,\!...,b_{n+1}], & 1\!\leq\! i\!\leq\! m,\\
a_0,a_1,\!...,a_{m}][b_0,...,b_{j}][b_j,...,b_{n+1}],& i\!=m\!+\!j\\
\end{array}
\right. \\
\\

 a_0,a_1,\!...,a_m][b_0,\!...,b_{n+1}]\overset{d^{1}_{i}}{\rightarrow}\!
 \left\{\!\!\! \!  \begin{array}{llll}
a_1,\!...,a_m][b_0,\!...,b_{n+1}][a_0,a_1], \  & && i=1,\\
a_0,a_1,\!...,\hat{a}_{i-1},\!...,a_m][b_0,\!...,b_{n+1}], \  & & &  2\!\leq \!i\!\leq\! m,\\

a_0,a_1,\!...,a_{m}][b_0,\!b_1,\!...,\hat{b}_{j},\!...,b_{n+1}], \  &&  &  i\!=m\!+\!j\\
\end{array}
\right.
 \\

 a_0,a_1,\!...,a_m][b_0,\!...,b_{n+1}]\!\overset{d^{2}_{i}}{\rightarrow}
a_i,\!...,a_{m}][b_0,\!...,b_{n+1}][a_0,a_1,\!...,a_i], \hspace{0.5in}  1\!\leq i\leq
m.
\end{array}
\]
 Then the
relations among $d^{\epsilon}$'s for $\epsilon=0,1,2$   agree with the $F_n$-set
identities.
\end{proposition}
\begin{proof}
It is straightforward.
\end{proof}

A degeneracy operator $ \eta _{i}$ is thought of as adding a formal element $\ast $ to
the set $a_0,a_1,...,a_m][b_0,b_1,...,b_{n+1}]$ at the $(m+i+2)^{st}$ place:
\begin{equation*}
a_0,a_1,...,a_m][b_0,b_1,...,b_{n+1}]\overset{\eta_i}{\longrightarrow}a_0,a_1,...,a_m]
[b_0,b_1,...,b_{i},\ast,b_{i+1},...,b_{n+1}]
\end{equation*}
with the convention  $[b_0,b_1,...,b_{i},\ast ][\ast
,b_{i+1},...,b_{n+1}]=[b_0,b_1,...,b_{n+1}]$ that guarantees the equality
$d_{m+i}^{0}\eta _{i}={Id}=d_{m+i}^{1}\eta _{i}$.

\vspace{0.2in}

\unitlength=1.00mm \special{em:linewidth 0.4pt} \linethickness{0.4pt}
\begin{picture}(112.33,43.33)
\put(9.67,40.00){\line(0,-1){20.00}} \put(9.67,20.00){\line(1,0){20.33}}
\put(30.00,20.00){\line(0,1){20.00}} \put(30.00,40.00){\line(-1,0){20.33}}
\put(9.67,40.00){\circle*{1.33}} \put(30.00,40.00){\circle*{1.33}}
\put(9.67,20.00){\circle*{1.33}} \put(30.00,20.00){\circle*{1.33}}
\put(30.00,30.00){\circle*{1.33}}

 \put(94.00,40.00){\line(2,-1){17.67}}
\put(111.67,31.00){\line(-1,0){22.00}} \put(89.67,31.00){\line(0,-1){20.33}}
\put(89.67,10.67){\line(1,0){22.00}} \put(111.67,10.67){\line(0,1){20.33}}
\put(89.67,10.67){\line(-2,1){17.67}} \put(72.00,19.67){\line(0,1){20.33}}
\put(72.00,19.67){\line(1,0){17.00}} \put(90.34,19.67){\line(1,0){3.67}}
\put(94.34,40.00){\line(0,-1){8.33}} \put(94.34,30.33){\line(0,-1){10.67}}
\put(94.34,19.67){\line(2,-1){17.33}} \put(102.67,35.33){\line(0,-1){4.00}}
\put(102.67,30.33){\line(0,-1){15.00}} \put(94.34,28.33){\line(5,-2){8.33}}
\put(72.00,40.00){\line(1,0){22.00}} \put(89.67,31.00){\line(-2,1){17.67}}
\put(94.34,40.00){\circle*{1.33}} \put(72.00,40.00){\circle*{1.33}}
\put(94.34,28.33){\circle*{1.33}} \put(102.67,25.00){\circle*{1.33}}
\put(102.67,35.67){\circle*{1.33}} \put(111.67,31.00){\circle*{1.33}}
\put(89.67,31.00){\circle*{1.33}} \put(94.34,19.67){\circle*{1.33}}
\put(102.67,15.33){\circle*{1.33}} \put(111.67,10.67){\circle*{1.33}}
\put(89.67,10.67){\circle*{1.33}} \put(72.00,19.67){\circle*{1.33}}

\put(6.00,42.50){\makebox(0,0)[cc]{$_{0][02]}$}}

\put(3.00,17.70){\makebox(0,0)[cc]{$_{0][01][12]}$}}

\put(35.00,42.00){\makebox(0,0)[cc]{$_{2][02]}$}}

\put(37.50,30.00){\makebox(0,0)[cc]{$_{2][01][12]}$}}

\put(37.00,17.70){\makebox(0,0)[cc]{$_{1][12][01]}$}}

\put(19.67,16.67){\makebox(0,0)[cc]{${01][12]}$}}
\put(19.67,43.33){\makebox(0,0)[cc]{${02]}$}}
\put(3.50,29.67){\makebox(0,0)[cc]{${0][012]}$}}
\put(37.00,35.00){\makebox(0,0)[cc]{${2][012]}$}}
\put(37.00,24.67){\makebox(0,0)[cc]{${12][01]}$}}

\put(99.00,33.67){\makebox(0,0)[cc]{$_{_{3]\![0123]}}$}}
\put(99.00,22.33){\makebox(0,0)[cc]{$_{_{23]\![012]}}$}} \ \
\put(107.60,23.67){\makebox(0,0)[cc]{$_{_{123]\![01]}}$}}
\put(90.67,35.00){\makebox(0,0)[cc]{$_{_{013]}}$}}
\put(83.34,30.33){\makebox(0,0)[cc]{$_{_{023]}}$}}

\put(83.34,17.33){\makebox(0,0)[cc]{$_{_{0]\![0123]}}$}}

\put(94.50,16.17){\makebox(0,0)[cc]{$_{_{012]\![23]}}$}}
\put(99.34,13.33){\makebox(0,0)[cc]{$_{_{01]\![123]}}$}}
\end{picture}
\vspace{-0.2in}
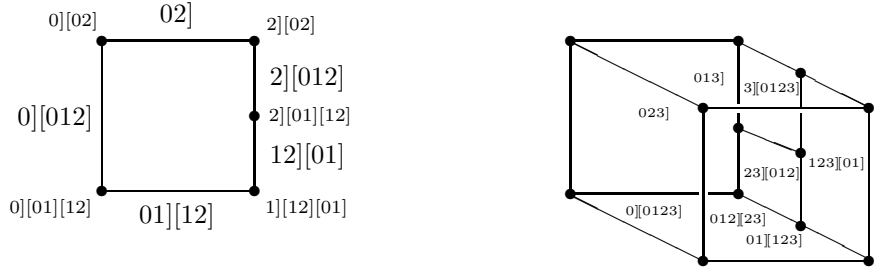
\begin{center}Figure 2. The combinatorial description of freehedra $F_2=012]$
and $F_3=0123].$
\end{center}

\section{A diagonal on the  freehedra $F_n$}\label{sdiagonal}

 For the freehedron  $F_n$ define its integral  chain
  complex
 $(C_*(F_n),d)$
with  differential
\begin{equation}\label{Fdiff}
d= \sum _{i=1}^n (-1)^{i}(d_i^0-d_i^1)+\sum _{i=2}^n (-1)^{(i-1)n}d^{2}_i.
  \end{equation}
Now define the map
$$
\Delta_F :C_*(F_n) \rightarrow C_*(F_n)\otimes C_*(F_n),
$$ for  $ u_n\in C_n(F_n), \ n\geq 0,$       by
\begin{multline}
\label{diagonal} \Delta_F (u_n)= \sum_{(K,L)} sgn (K,L)\, d^0_{j_p}\!...
d^0_{j_1}(u_n)\otimes
   d^{1}_{i_q} \!... d^1_{i_1}(u_n)+ \\
 \sum_{\substack{_{(K',L')}\\_{r+\!1\leq i_q\leq p+\!1 }\\
_{0\leq r<n}}} \!\!\!\!   (-1)^{{j_{(r)}+r+(p+1)(i_q+1)}}\!\cdot\! sgn (K',\!L')\,
 d^0_{j_p}\!...
d^{0}_{j_{r+1}}d^{2}_{j_{r}}\!... d^2_{j_1}(u_n)\otimes
   d^{2}_{i_q} d^1_{i_{q-1}}\!... d^1_{i_1}(u_n),\\
\hspace{-1.65in}
 (K,L)=\, \left(i_{q}<...<i_1\,; 1=j_{p}<...<j_{1}\right)  \ \  \text{and}\ \ \\
 \hspace{0.1in}
 (K',L') =
 \left(1<i_{q-1}\!<...<\!i_1\,; j_1\!+\!1\!<...< \!j_{(r)}\!+\!1<\!j_p\!+\!
j_{(r)}\!<...<\! j_{r+1}\!+\!j_{(r)}\right)
\end{multline}
 with  $j_{(k)}= j_1+\cdots
+j_k$ are unshuffles  of the set $\{1,...,n\}.$

\begin{proposition}
The map defined by formula (\ref{diagonal}) is a chain map.
\end{proposition}
\begin{proof}
The proof is  straightforward.
\end{proof}
Note that the components of the first summand of   (\ref{diagonal}) together with ones
of the second for $(r,i_q)=(0,1)$ agree with the components of the diagonal of the
standard cube $I^n$ \cite{Serre}, \cite{KScubi}.

Using Proposition \ref{Fcubi}  formula (\ref{diagonal}) can be rewritten in the
following combinatorial form (compare (\ref{combiserre})):
\begin{multline}\label{combidiagonal}
\Delta_F(\, 0,1,...,n])= \\
 \hspace{0.8in} \Sigma (-1)^{\epsilon_1 }\
0,1,...,j_{1}][j_{1},...,j_{2}][j_{2},...,j_{3}]...[j_{p},...,n]\otimes
 j_{1},j_{2},...,j_{p},n]+
\\
\ \  \Sigma(-1)^{\epsilon_2}j_{r},...,j_{r+1}][j_{r+1},...,j_{r+2}] ...[j_{p},...,n]
  [0,1,...,j_{1}][j_1,...,j_2]...[j_{r-1},...,j_r]
\otimes\\ j_t,...,j_p,n][0,j_1,...,j_r,...,j_{t}],
\end{multline}
where the last tensor factor for a fixed $r$ varies from
$j_{r+1},...,j_p,n][0,j_1,...,j_{r+1}]$ to
 $n][0,j_1,...,j_p,n].$

 For example, formulas (\ref{diagonal}) and
(\ref{combidiagonal}) for $F_2$ and $F_3$ read:
\begin{multline*}
\Delta_F (u_2) =\left (d^0_1d^0_2\otimes Id+Id\otimes d^1_1d^1_2 -d^0_1\otimes d^1_2 +
 d^0_2\otimes  d^1_1 \right.+\\
 \left. (d^0_2+
d^2_1)\otimes  d^2_2\right) (u_2\otimes u_2);
\end{multline*}
\begin{multline*}
\Delta_F(u_3) =\left (d^0_1d^0_2d^0_3\otimes Id+Id\otimes d^1_1d^1_2d^1_3+
 d^0_1\otimes d^1_2d^1_3-
d^0_2\otimes d^1_1 d^1_3+ d^0_3\otimes d^1_1d^1_2\right.+\\
\ \ \ \ \ \ \ \ \ \ \ \ \ \ \ \ \ \ \ \ \ \ \ \ \ \ \ \ \ \ \ \ \ \ \ \ \ \ \ \ \ \ \ \
\ \ \ \ \ \  \ \ \ \ \
 d^0_1d^0_2\otimes d^1_3-
  d^0_1d^0_3\otimes d^1_2
+d^0_2d^0_3\otimes d^1_1 -\\
 \ \ \ \ \ \ \ \ \ \ \ \ \ \ \ \
 (d^0_2+d^2_1)\otimes d^2_2 d^1_3+(d^0_3
 +d^2_2)\otimes d^2_2d^1_2
- ( d^0_2d^0_3+
 d^0_2d^2_1)\otimes (d^2_2-d^2_3)+ \\
 \left. d^0_2d^2_2\otimes d^2_3
 \right) (u_3\otimes u_3)
\end{multline*}
and
\begin{multline*}
\!\!\! \!\Delta_F \left(\,012]\right)= 0][01][12]\otimes 012]+012]\otimes 2][02] -
0][012]\otimes 02] +
 01][12]\otimes  12][01]+\\
\left(\, 01][12]  + 12][01]\right)\otimes  2][012];
\end{multline*}
\begin{multline*}
\!\!\! \!\Delta_F\left(\,0123]\right)=\\
\ \ \ \ \ \ \ \ \ \ \
 0][01][12][23]\otimes 0123]+  0123]\otimes 3][03]+
 0][0123]\otimes 03]-
01][123]\otimes 13][01] +\\
\ \ \ \ \
 012][23]\otimes 23][02]+
 01][12][23]\otimes 123][01]+
 0][01][123]\otimes 013]-0][012][23]\otimes 023]-\\
\ \ \ \ \ \ \ \ \ \ \ \ \ \ \ \ \ \ \ \ \ \ \ \ \ (\, 01][123]+123][01])\otimes
3][013]+(\,012][23] +23][012])\otimes 3][023]
-\\
 \ \ \ \ \ \ \ \ \ \ \ \ \ \ \ \ \ \ \ \ \ \ \ \ \ \ \ \ \  \ \ \ \ \ \ \ \ \ \ \ \ \
(\,01][12][23]+12][23][01])\otimes (\,23][012]-3][0123] )+\\
23][01][12]\otimes 3][0123] .
\end{multline*}
The diagonal $\Delta_F$ is compatible with the  AW diagonal of the standard simplex
$\Delta^n$ under the cellular map $\varphi :F_n\rightarrow \Delta^n.$ To see this it is
also convenient to represent $\varphi$ combinatorially as:
\begin{multline*}
\!\!\!\!      i_{s_t},\!...,i_{s_{t+1}}][i_{s_{t+1}},\!...,i_{s_{t+2}}]
...[i_{s_{k-1}},\!...,i_{s_{k}},\!n] [0,i_1,\!...,i_{s_1}][i_{s_1},\!...,i_{s_{2}}]
...[i_{s_{t-1}},\!...,i_{s_t}] \overset{\varphi}{\rightarrow} \\
\left(i_{s_t},\!...,i_{s_{t+1}}\right).
\end{multline*}
 In particular, the faces $0][0,1,...,n]$ and $n][0,1,...,n]$ of $F_{n}$,
  i.e. $ d_{1}^{0}$ and $d^2_n,$ go to the
minimal  and maximal vertices   $0\in \Delta ^{n}$ and $n\in \Delta ^{n}$ respectively
(see Figure 3).

Note that  the  diagonal  $\Delta_F$ on $C_*(F_n)$  is not coassociative, and hence,
the product on $C^*(F_n)$
   is not associative, however since the acyclicity of $F_n$  there exists an
   $A_{\infty}$-algebra
   structure on $C^*(F_n)$ (see   Subsection \ref{interaction} below).

\subsection{The diagonal on an $F_n$-set}
Given an $F_n$-set $\mathcal{CH},$   define the chain complex $(C_*(\mathcal{CH}),d)$
of $\mathcal{CH}$ with coefficients in $\Bbbk$ and with differential
$d_r:C_{r}(\mathcal{CH})\rightarrow C_{r-1}(\mathcal{CH})$  given  by
$$d_r=\bigoplus_{\substack{{r=m+n}\\{m,n\geq 0}}}\, d_{m,n},\ \ \    d_{m,n}= \sum _{i=1}^{m+n} (-1)^{i}(d_i^0-d_i^1)+\sum _{i=2}^m (-1)^{(i-1)(m+n)}d^2_i.
   $$
  The \emph{normalized chain
complex} of $\mathcal{CH}$ is
 $(C^{\odot}_{*}(\mathcal{CH}),d)= (C_*(\mathcal{CH}),d)/D,$
where  $D$ is the subcomplex of $C_*(\mathcal{CH})$ formed by degeneracies. Then apply
(\ref{diagonal}) to make $C^{\odot}_{*}(\mathcal{CH})$ as a dg \emph{coalgebra}. In
particular, given the singular $F_n$-set $\operatorname{Sing}^FY,$ we get the dg
coalgebra  $C^{\odot}_{*}(\operatorname{Sing}^FY)$ denoted by
    $(C^{\odot}_{*}(Y),d).$
 Then the cellular composition $$I^k\times F_m\times I^{\ell}\overset{1\times\phi\times
1}{\longrightarrow}I^k\times I^m\times I^{\ell}  =
I^{k+m+\ell}\overset{\psi}{\rightarrow} \Delta^{k+m+\ell}$$ induces a chain map
$C_*(Y)\rightarrow C^{\odot}_{*}(Y)$
 to obtain the
following
\begin{proposition}\label{Fcoalgebra}
There are  the natural isomorphisms  of the homologies
\[ H_*(Y)\approx  H^{\odot}_*(Y)=H_*(C^{\odot}_*(Y),d)\]
and the cohomology algebras
 \[H^*(Y)\approx  H_{\odot}^*(Y)=H^*(C_{\odot}^*(Y),d).
\]
\end{proposition}

%\vspace{0.2in}
\section{Truncating twisting functions and bitwisted Cartesian products}

\subsection{The Cartier-Hochschild set $\mathbf{\Lambda} X$}

 Given
a 1-reduced simplicial set $(\!X\!,\partial_i,s_i)\!,$  i.e. $X=\{X_{0}=X_{1}=\{\ast
\},X_{2},X_{3},... \},$ recall the definition of
 a \emph{truncating twisting
function} \cite{KScubi}\,:

\begin{definition}
Let $X$ be a 1-reduced simplicial set and $Q$ be a monoidal cubical set. A sequence of
functions $\tau =\{\tau _{n}:X_{n}\rightarrow Q_{n-1}\}_{n\geq 1} $ of degree $-1$ is a
\underline{truncating twisting function} if it satisfies:
\begin{equation*}
\begin{array}{lll}
\hspace{0.14in} \tau(x) =e, &   x\in X_1,$\newline
$\vspace{1mm} \\
d^0_i \tau (x)= \tau \partial_{i+1}...\, \partial_{n}(x)\cdot \tau
\partial_{0}\,...\, \partial_{i-1}(x),  & x\in X_n, &  1\leq i<n, $
\newline
$\vspace{1mm} \\
d^1_i\tau (x)=\tau\partial_{i}(x), &  x\in X_n, & 1\leq i<n,$
\newline
$\vspace{1mm} \\
\eta_{n}\tau(x)=\tau s_n(x), &   x\in X_{n}, & n\geq 1 .
\end{array}
\end{equation*}
\end{definition}
A useful  characterization  of a truncating  twisting function is that the monoidal map
$f:{\mathbf \Omega} X\rightarrow Q$ defined by $f(\bar{x}_{1}\cdots \bar{x}_{k})=\tau
(x_{1})\cdots \tau (x_{k})$ is a map of cubical sets, where ${\mathbf \Omega} X$ is the
monoidal cubical set constructed in  \cite{KScubi} such that ${\mathbf \Omega} X$ is
related with $X$ by the universal truncating twisting function $\tau_{_{U}}:
X\rightarrow {\mathbf \Omega} X, $\, $x \rightarrow \bar x.$

\begin{definition}
\label{tmodule} Let $X=\{X_m,\partial_i,s_i\}$ be a 1-reduced simplicial set, $Q$ be a
monoidal cubical set, and $L=\{L_n,d^{\epsilon}_i,\eta_i\}$ be a $Q$-bimodule via
$Q\times L\rightarrow L$ and $L\times Q\rightarrow L.$ Let $\tau =\{\tau
_{k}:X_{k}\rightarrow Q_{k-1}\}_{k\geq 1}$ be a truncating twisting function. The
\underline{bitwisted Cartesian product $X\, _{\tau}\!\!\times _{\tau }L$} is the
bigraded set
\[X\, _{\tau}\!\!\times _{\tau }L=X\times L/_{\sim},\] where $(s_m(x),y)\sim
(x,\eta_1(y)),$ \,$(x,y)\in X_{m}\times L_{n},$ and
 endowed with the face $d_{i}^{0},d_{i}^{1},d^2_i$ and degeneracy $ \eta _{i}$
operators defined   by
\begin{equation}\label{F-structure}
\begin{array}{llll}
d^0_i (x,y ) =\left\{\!\!\!
\begin{array}{lll}
(\partial_{1}...\, \partial_{m}(x)\,, \tau (x)\!\cdot \!y) , &\ \ \ \  i=1, \\
( \partial_{i}...\, \partial_{m}(x)\,, \tau \partial_{0}...\,
\partial_{i-2}(x)\!\cdot\! y) , & \ \ \ \  1<i \leq m, \\
(x,d^0_{i-m}(y)), & \ \ \ \  m<i\leq m+n ,
\end{array}
\right.
 $\newline
$\vspace{0.1in}  \\
d^1_i ( x,y ) =\left\{\!\!
\begin{array}{lll}
(\partial _{i-1}(x),y), & \ \ \ \ \ \ \ \ \ \ \ \ \ \ \ \ \ \ \ \ \ \ \ \ \,
1\leq i\leq m, \\
( x, d^1_{i-m}(y)), & \ \ \ \ \ \ \ \ \ \ \ \ \ \ \ \ \ \ \ \ \ \ \ \ \,   m<i\leq m+n,
\end{array}
\right.
\vspace{1mm} \\
d^2_i(x,y)=\,\,\left(\, \partial_{0}...\, \partial_{i-1}(x)\,,\, y\!\cdot\! \tau
\partial_{i+1}... \,\partial_{m}(x)\right),\ \ \ \   1\leq i \leq m,
\\
\\
\, \eta_i(x , y)=\, \left(x,\eta_i(y)\right),\hspace{1.62in}   1\leq i\leq n+1.
\end{array}
\end{equation}
\end{definition}
Using Proposition \ref{Fcubi} it is easy to verify  that
  $(X\,_{\tau}\!\!\times _{\tau}L,\
d^0_i,d^1_i,d^2_i,\eta_i) $ forms an $F_n$-set. In particular,
$$d^2_1(x,y)=(\partial_0(x)\,,y\!\cdot\! \tau
\partial_2...\,\partial_{m}(x))=(\partial_0(x)\,,y\!\cdot\! e)=(\partial_0(x)\,,y)=d^1_1(x,y).$$

Take in the above definition $Q=L=\mathbf{\Omega} X ,$ $\tau=\tau_{_{U}},$ and the
bimodule actions inducing by the monoidal product on the cubical set $\mathbf{\Omega}
X$ to give the following
\begin{definition}
Given a 1-reduced simplicial set $X,$ \underline{the Cartier-Hochschild set
$\mathbf{\Lambda} X$} is the bitwisted Cartesian product $X\,
_{\tau\!_{_{U}}}\!\!\!\times_{\tau\!_{_{U}}} \mathbf{\Omega} X $ endowed with the
$F_n$-set structure via (\ref{F-structure}).
\end{definition}
\vspace{0.2in}
 \unitlength=1.00mm \special{em:linewidth 0.4pt} \linethickness{0.4pt}
\begin{picture}(118.32,90.00)
\put(65.33,74.67){\line(-1,0){24.33}} \put(41.00,74.67){\line(1,1){13.67}}
\put(41.00,74.67){\line(0,-1){22.33}} \put(41.00,52.34){\line(1,0){24.67}}
\put(65.33,74.67){\line(0,-1){22.33}} \put(55.33,88.34){\line(0,-1){12.67}}
\put(55.33,66.00){\line(1,0){9.00}} \put(2.00,74.34){\line(0,-1){22.00}}
\put(2.00,52.34){\line(1,1){13.67}} \put(15.67,66.00){\line(0,1){22.33}}
\put(41.00,16.34){\line(3,2){19.00}} \put(60.00,29.00){\line(3,-2){19.00}}
\put(79.00,16.34){\line(-3,-2){19.00}} \put(60.00,3.67){\line(-3,2){19.00}}
\put(41.00,16.34){\line(1,0){16.67}} \put(60.00,28.67){\line(0,-1){25.00}}
\put(61.67,16.34){\line(1,0){16.67}} \put(2.00,53.00){\circle*{1.33}}
\put(2.00,74.34){\circle*{1.33}} \put(15.67,65.67){\circle*{1.33}}
\put(41.00,74.67){\circle*{1.33}} \put(55.00,88.00){\circle*{1.33}}
\put(65.33,74.67){\circle*{1.49}} \put(55.33,66.00){\circle*{1.33}}
\put(65.33,52.34){\circle*{1.33}} \put(60.00,28.34){\circle*{1.33}}
\put(41.00,16.34){\circle*{1.33}} \put(78.67,16.34){\circle*{1.33}}
\put(60.00,4.34){\circle*{1.33}} \put(15.67,88.34){\circle*{1.33}}
\put(2.00,74.67){\line(1,1){13.67}} \put(41.00,52.34){\circle*{1.33}}

\put(20.00,70.67){\vector(1,0){15.67}}

\put(27.00,73.50){\makebox(0,0)[cc]{$\delta^0_1$}}

\put(98.33,79.67){\vector(-1,0){20.00}}

\put(90.00, 82.50){\makebox(0,0)[cc]{$\delta^2_3$}}

 \put(60.00,47.34){\vector(0,-1){13.00}}
\put(56.00,40.00){\makebox(0,0)[cc]{$\varphi$}}

\put(44.33,26.67){\vector(-1,1){29.00}}

\put(29.33,46.00){\makebox(0,0)[cc]{$\tau\!_{_{U}}$}}
\put(90.33,46.33){\makebox(0,0)[cc]{$\tau\!_{_{U}}$}}

\put(37.67,16.00){\makebox(0,0)[cc]{${0}$}}

\put(62.00,31.00){\makebox(0,0)[cc]{$3$}}

 \put(82.33,16.34){\makebox(0,0)[cc]{$2$}}

 \put(4.67,51.60){\makebox(0,0)[cc]{$0$}}

\put(5.00,74.33){\makebox(0,0)[cc]{$0$}}

\put(18.33,88.67){\makebox(0,0)[cc]{$0$}}

 \put(18.33,65.67){\makebox(0,0)[cc]{$0$}}

\put(41.00,49.34){\makebox(0,0)[cc]{$_{0}$}}
 \put(65.33,49.34){\makebox(0,0)[cc]{$_{1}$}}
\put(52.67,66.00){\makebox(0,0)[cc]{$_{0}$}}
 \put(52.33,88.00){\makebox(0,0)[cc]{$_{0}$}}
\put(64.67,76.67){\makebox(0,0)[cc]{$_{1}$}}
 \put(39.00,74.67){\makebox(0,0)[cc]{$_{0}$}}
\put(62.00,90.20){\makebox(0,0)[cc]{$_{3}$}}

\put(70.66,86.00){\makebox(0,0)[cc]{$_{3}$}}

 \put(67.46,79.40){\makebox(0,0)[cc]{$_{3}$}}

\put(76.66,76.33){\makebox(0,0)[cc]{$_{3}$}}

 \put(76.00,60.00){\makebox(0,0)[cc]{$_{2}$}}
\put(70.33,64.00){\makebox(0,0)[cc]{$_{2}$}}

 \put(77.01,6.67){\makebox(0,0)[cc]{$(0123)$}}
\put(9.33,69.00){\makebox(0,0)[cc]{$_{0][0123]}$}}

 \put(69.66,79.00){\line(-5,6){7.67}}
\put(62.00,88.00){\line(-1,0){6.67}} \put(69.66,79.00){\line(-1,-1){4.00}}
\put(70.66,66.00){\line(1,-1){4.00}} \put(74.66,62.00){\line(0,1){14.33}}
\put(69.66,79.00){\line(5,-3){5.00}} \put(62.00,88.00){\line(2,-1){8.33}}
\put(70.33,83.67){\line(3,-5){4.33}} \put(70.66,66.00){\line(0,1){11.67}}
\put(70.66,79.00){\line(0,1){4.33}} \put(66.00,66.00){\line(1,0){4.67}}
\put(74.66,62.33){\line(-5,-6){8.33}} \put(41.00,52.00){\line(1,1){14.33}}
\put(62.00,88.00){\circle*{1.33}} \put(70.66,83.67){\circle*{1.33}}
\put(69.33,79.00){\circle*{1.33}} \put(74.66,76.33){\circle*{1.33}}
\put(70.66,66.00){\circle*{1.33}} \put(74.66,62.00){\circle*{0.00}}
\put(74.66,62.00){\circle*{1.33}}

 \put(74.33,26.33){\vector(1,1){30.00}}

\put(110.66,70.33){\makebox(0,0)[cc]{$_{3][0123]}$}}
\put(77.33,54.67){\makebox(0,0)[cc]{$0123]$}}

 \put(103.00,65.67){\line(0,1){21.33}}
\put(103.00,87.00){\line(4,-3){14.67}} \put(117.66,55.00){\line(-4,3){14.67}}
\put(117.66,55.33){\line(0,1){20.67}} \put(117.66,55.33){\circle*{1.33}}
\put(103.00,65.67){\circle*{1.33}} \put(103.00,86.67){\circle*{0.00}}
\put(103.00,86.67){\circle*{1.33}} \put(117.66,76.00){\circle*{1.33}}

 \put(115.67,54.00){\makebox(0,0)[cc]{$3$}}
\put(100.33,65.66){\makebox(0,0)[cc]{$3$}} \put(100.33,86.66){\makebox(0,0)[cc]{$3$}}
\put(115.67,75.33){\makebox(0,0)[cc]{$3$}} \put(55.33,66.00){\line(0,1){8.00}}
\put(62.50,3.00){\makebox(0,0)[cc]{$1$}}
\end{picture}

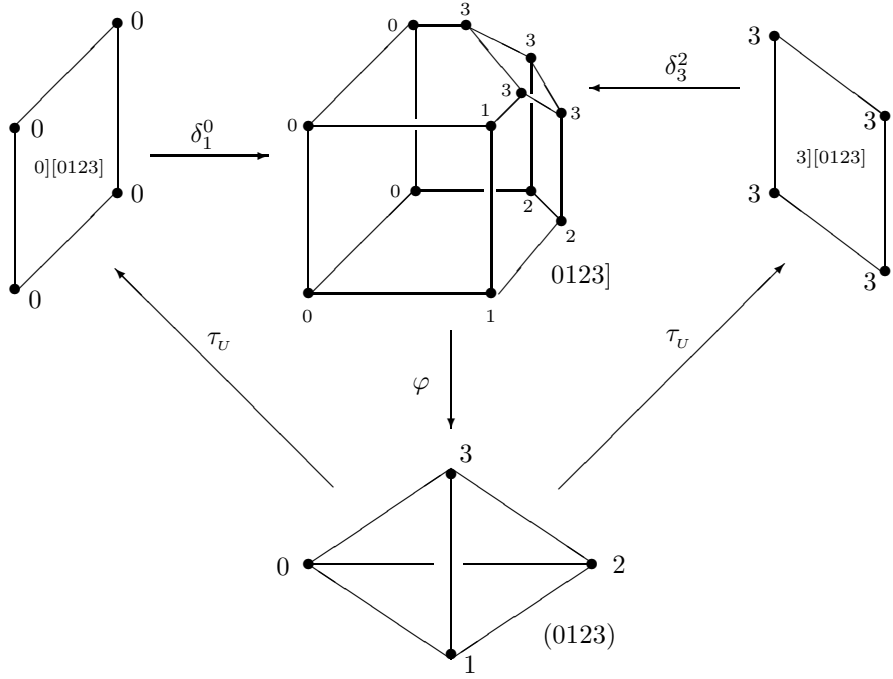
\begin{center}{Figure 3. The two-fold interpretation of  the universal truncating
twisting \\ function $\tau\!_{_{U}}.$} \end{center}

\vspace{1mm}

\begin{remark}\label{PX}
1.  Note that  in the  Cartier-Hochschild  set $\mathbf{\Lambda} X$  one implies
 the identity  $d^0_1(x,e)=d^2_{m}(x,e)=\tau_{_{U}}(x)=\bar x   $
 for any simplex $x \in X_{m}.$

2. The operators $d^0, d^1$ just subject to the defining identities of a cubical set.

\end{remark}

\section{The bitwisted Cartesian model for the free loop  fibration}\label{free}

Let   $\Omega Y \rightarrow \Lambda Y\overset{\xi}{\rightarrow} Y$ be  the free loop
space fibration on a topological space $Y.$ Let ${\operatorname{Sing}}^{1}Y\subset
{\operatorname{Sing}}Y$ be the Eilenberg 1-subcomplex generated by the singular
simplices that send the 1-skeleton of the standard $n$-simplex $\Delta ^{n}$ to the
base point $y$ of $Y.$
 Denote by  $ C_*(Y)$
 the quotient coalgebra $C_*(\operatorname{Sing}^1 Y)/C_{>0}(\operatorname{Sing}\, y),$
 the  chain complex of $Y.$
Let $\operatorname{Sing}^I \Omega Y $ be the singular cubical set of $\Omega Y$ and
$C_*^{\Box}(\Omega Y)$ be the (normalized) chain complex of
  $\operatorname{Sing}^I \Omega Y. $
 Then Adams' map $\omega _{\ast }:\Omega C_{\ast }(Y)=C_{\ast
}({\mathbf\Omega} \operatorname{Sing} ^1 Y)\rightarrow C_{\ast }^{\Box }(\Omega Y)$ is
realized by a monoidal cubical map $\omega :{\bf\Omega} \operatorname{Sing} ^1 Y
\rightarrow {\operatorname{Sing}} ^{I}\Omega Y$ \cite{KScubi}.
 Obviously, we have  the  short sequence of singular $F_n$-sets
$$\operatorname{Sing} ^{F}   \Omega Y
\longrightarrow \operatorname{Sing} ^{F}  {\Lambda} Y
\overset{\xi_{_{\#}}}{\longrightarrow} \operatorname{Sing} ^{F}   Y. $$ On the other
hand, there is the short sequence of  sets
$$\mathbf{\Omega} \operatorname{Sing} ^1 Y \rightarrow  \mathbf{\Lambda}
\operatorname{Sing} ^1 Y \overset{p}{\rightarrow} \operatorname{Sing} ^1 Y$$ where the
maps are the natural inclusion and  projection respectively. Let $$\iota:
\operatorname{Sing} ^I \Omega Y \rightarrow \operatorname{Sing}^{F} \Omega Y$$ be the
inclusion determined via the identification  $ (\operatorname{Sing} ^I \Omega Y)_* =(
\operatorname{Sing}^{F} \Omega Y)^{0, *}.$

We have the following

\begin{theorem} \label{main}
Let
 $\Omega Y \rightarrow \Lambda Y\overset{\xi}{\longrightarrow} Y$
be the free loop fibration.

(i) There are natural maps of  sets $\tilde\varphi ,\Upsilon,\tilde\omega $ such that
\begin{equation}\label{CHmodel}
\begin{array}{ccccc}
\operatorname{Sing} ^{F}  \Omega Y \longrightarrow  &\operatorname{Sing} ^{F} \Lambda Y
  \overset{\xi_{_{\#}}}{\longrightarrow} & \operatorname{Sing} ^{F} Y \\
\tilde \omega \uparrow        &{\Upsilon}\uparrow            &\tilde\varphi \uparrow\\
 {\mathbf \Omega} \operatorname{Sing} ^1
Y \longrightarrow  &  \mathbf{\Lambda} \operatorname{Sing} ^1 Y
\overset{p}{\longrightarrow} & \operatorname{Sing} ^1 Y,
\end{array}
\end{equation}
$\tilde\varphi :{\operatorname{Sing}}^{1}Y\rightarrow {\operatorname{Sing}}^{F}Y$ is a
map  induced by the composition $I^k\times F_m\times I^{\ell}\overset{1\times\phi\times
1}{\longrightarrow}I^{k+m+\ell}\overset{\psi}{\rightarrow} \Delta^{k+m+\ell},$  while
$\Upsilon$ is a map of $F_n$-sets, and $\tilde \omega=\iota\circ \omega;$ moreover,
these maps are homotopy equivalences whenever $Y$ is simply connected.

(ii) The chain complex $C_{\ast }^{\odot}(\mathbf{\Lambda} \operatorname{Sing}^{1}Y)
 $ of the bitwisted Cartesian product $$\mathbf{\Lambda} \operatorname{Sing}^{1}Y=
 \operatorname{Sing}^{1}\!Y\,_{\tau_{_{U}}}\!\!\!\times_{\tau_{_{U}}}
 \mathbf {\Omega }\operatorname{Sing}^{1}\!Y$$
  coincides with the Cartier complex $\Lambda C_{\ast}(Y)$ of the  chain coalgebra
 $C_{\ast}(Y).$
\end{theorem}
\begin{proof}
To define the map $\Upsilon :\mathbf{\Lambda} \operatorname{Sing}^1Y \rightarrow
\operatorname{Sing}^{F} \Lambda Y,$ it is convenient to apply a homotopically
equivalent description of the free loop fibration $\xi$ (cf. \cite{McCleary}). Namely,
$\xi$ is thought of as the associated fibre bundle with the universal bundle
$G\rightarrow EG\rightarrow BG$ via the conjugation action $G\times G\rightarrow G,\
(a,b)\rightarrow a^{-1}ba, $ where $G$ has the homotopy type of  $\Omega Y$ and let
$\pi:EG\times G\rightarrow (EG\times G)/_{\sim}= \Lambda BG$ be the quotient map.

Fix a section $s:Y\rightarrow \Lambda BG.$ Choose its factorization $s: Y
\overset{s'}{\rightarrow} EG\times G\overset{\pi}{\rightarrow} \Lambda BG$ ($s'$ does
not need to be continuous). Fix a homotopy $\chi_{k}:I^k\times I\rightarrow I^k$
contracting $I^k$ at the minimal vertex. Let $\rho: F_m\rightarrow I^{m}=I^{m-1}\times
I$ be a cellular projection obtained
 by dilatation of the face $d^2_m(F_m)$ to
$d^1_{1}(I^m).$ Given $(\sigma_m,\sigma'_n)\in (\mathbf{\Lambda}
\operatorname{Sing}^1Y)^{m,n},$ let $\Upsilon(\sigma_m,\sigma'_n)\in
(\operatorname{Sing}^{F}Y)^{m,n}$ be a  map $f:F_m\times I^n\rightarrow \Lambda BG$
defined  as follows: Consider the compositions
\[
\begin{array}{lll}
\zeta: F_m\times I^n\overset{\rho\times 1}{\longrightarrow} I^{m-1}\times I \times I^n
 \overset{p_1}{\longrightarrow}
 I^{m-1}\times I^n\overset{\tilde\omega (\bar \sigma_m\!\cdot \sigma'_n)}{\longrightarrow}
 G, \\
h: F_m\times I^n\overset{\rho\times 1}{\longrightarrow} I^{m-1}\times I \times I^n
 \overset{p_2}{\longrightarrow}
 I^{m-1}\times I\overset{\chi_{m\!-\!1}}{\longrightarrow} I^{m-1}
 \overset{\tilde\omega (\bar \sigma_m)}{\longrightarrow}
 G,
 \end{array}
\]
  where $p_1$ and $p_2$ are
 canonical projections.
 Given $(u,v)\in F_m\times I^n,$ let
\begin{equation*}\label{action}
f(u,v)= \pi\left(x_u\!\cdot\!h(u,v) \,,\, y_u\!\cdot\! \zeta(u,v)\right),
\end{equation*}
 where $(x_u,y_u)= (\varphi\circ \sigma_m\circ s')(u)\in EG\times G. $
 Since $ \varphi\circ \sigma_m\circ s: F_m\rightarrow \Lambda BG$ is continuous,
 it is easy to verify that $f$ will be continuous as well.
The relation $(xa,ab)\sim (x,ba)$ in $(EG\times G)/_{\sim}$ guarantees the
 compatibility  $d^2_i(\sigma_m,\sigma'_n)$ with $d^2_i(f)$  under $\Upsilon.$
Thus, for $m=0,1$  the map $\Upsilon (\sigma_m ,e):F_m \to \Lambda BG$
 is constant to
the base point $\lambda_0=\pi(x_0,e)$, where $e$ denotes
 the unit of the monoid $\operatorname{Sing} ^I \Omega Y$ too, and $x_0$ is the base
 point of $EG.$ On the other hand,  we have $\Upsilon (\sigma_0 ,\sigma'_n)
 :I^n \overset{\tilde\omega
 (\sigma'_n\!)} {\longrightarrow} G \hookrightarrow \Lambda BG$ for $n\geq 0.$

The proof of $\Upsilon$  being a homotopy equivalence (after the geometric realization)
immediately follows, for example, from the comparison of the standard spectral
sequences for $\xi$ and $p$ using the fact that $\tilde \omega_*$ and $\tilde
\varphi_*$ are homology isomorphisms.

(ii) The proof is straightforward as the proof of the identification isomorphism
$C_{\ast }^{\Box }({\bf P} \operatorname{Sing}^{1}Y) =\Omega \left(C_{\ast }(Y);C_{\ast
}(Y)\right)$ in \cite{KScubi}. Only we remark that the differential of $\Lambda C_{\ast
}(Y)$ differs from the one of the above acyclic cobar construction by the component
$\theta_2$ that agrees with the  component  $d^2$ of the differential $d$ of the chain
complex $C_{\ast }^{\odot}(\mathbf{\Lambda} \operatorname{Sing}^{1}Y)$ under the
required identification here.
 \end{proof}

Thus, by passing to chain complexes in diagram (\ref{CHmodel}) we obtain the following
comultiplicative model of the free loop fibration $\xi$ formed by dgc's.

\begin{theorem}\label{hohmodel}

 For the free loop space fibration
 $\Omega Y \rightarrow \Lambda Y\overset{\xi}{\rightarrow} Y$
there is a comultiplicative model formed by dgc's which is natural in $Y:$
\begin{equation}
\begin{array}{ccccc}
C^{\odot}_*(  \Omega Y) &\longrightarrow & C^{\odot}_*( \Lambda Y)&
\overset{\xi_*}{\longrightarrow}
 & C^{\odot}_*( Y )   \\
\tilde \omega_* \uparrow  & &\Upsilon_* \uparrow    &  &\tilde \varphi_{*} \uparrow  \\
\Omega C_*( Y) & \longrightarrow  &\Lambda C_*( Y ) & \overset{p_*}{\longrightarrow} &
C_*( Y).
\end{array}
\end{equation}
\end{theorem}
Obviously, one obtains the dual  statement for cochain complexes involving the
Hochschild dg algebra $\Lambda C^*(Y).$

\subsection{The canonical homotopy $G$-algebra structure on $C^*(X)$}
Recall that there exists a canonical hga structure $\{E_{k,1}:C^*(X)^{\otimes k}\otimes
C^*(X)\rightarrow C^*(X)\}_{k\geq 0}$ on the simplicial cochain algebra $C^*(X)$
\cite{Baues3}, \cite{Getz-Jones}, \cite{KScubi} which, in particular, defines an
associative multiplication $\mu_E$ on the bar construction $BC^*(X).$ It is convenient
to view these operations as the dual of the cooperations $E^{k,1}$ on the simplicial
chain coalgebra $C_*(X).$ In turn, these cooperations can be obtained by a
combinatorial analysis of the diagonal of $I^n:$
\begin{multline}\label{combiserre}
\Delta \lbrack 0,1,...,n+1]=  \Sigma (-1)^{\epsilon }\
[0,1,...,j_{1}][j_{1},...,j_{2}][j_{2},...,j_{3}]...[j_{p},...,n+1]\otimes
\\
 [0,j_{1},j_{2},...,j_{p},n+1].
\end{multline}
where the summands $[01...n+1]\otimes \lbrack 0,n+1]$ and $
[01][12][23]...[n,n+1]\otimes \lbrack 01...n+1]$ form the primitive part of the
diagonal.

Regard the blocks of natural numbers above as faces of the standard $ (n+1)$-simplex
and discard the expression of the form $[j,j+1] $ to obtain Baues' formula for a
1-reduced simplicial set $X$ and  a generator $\sigma \in C_{n+1}(X):$
\begin{multline}\label{cohga}
E^{k,1} ( \sigma )= \Sigma (-1)^{\epsilon }\left( \sigma (0,1,\!...,j_{s_1})\!\otimes
\!\sigma (j_{s_2},\!...,j_{s_3})\!\otimes \!\cdots \!\otimes\!\sigma
(j_{s_k},...,n+1)\right)\!\otimes\!  \\
 \sigma (0,j_{1},j_{2},...,j_{p},n+1),
\end{multline}
where $\sigma (i_{1},...,i_{r})$ denotes the suitable face of $\sigma $ (i.e.
\!$[i_{1},...,i_{r}]=\tau_{_{U}}\sigma (i_{1},...,i_{r})$) and  $k\leq p.$

Now let $\lambda_E$ denote an inducing  multiplication   by (\ref{diagonal})  on the
complex $\Lambda C^*(X).$ Apply formulas
(\ref{combidiagonal}),\,(\ref{combiserre}),\,(\ref{cohga}) to write down $\lambda_E$ in
terms of operations $E_{k,1}.$  Namely,
 given two elements $u\otimes [\bar a_1|\dotsb |\bar a_n]$  and $v\otimes [\bar
b_1|\dotsm |\bar b_m]$ in $\Lambda  C^{*}(X),$  we get
\newpage
\begin{multline}\label{formula}
\lambda_{E}\left(\left(u\otimes [\bar a_1|\dotsm |\bar a_m]\right)\otimes
\left(v\otimes [\bar b_1|\dotsm |\bar b_n]\,\right)\right)= \\
 \sum _{p=0}^{m}\,
 (-1)^{\varepsilon_1}\,
u\!\cdot\!E_{p,1 }( a_1,\!... , a_p; v) \otimes \mu_E\!\left( [\bar a_{p+1
}|\dotsb |\bar a_m] \otimes [\bar b_1|\dotsb |\bar b_n]\right)+\\
 \sum_{0\leq i\leq j\leq k\leq m}(-1)^{\varepsilon_2}
 E_{m+i+1-k,1}( a_{k+1}
, \!... , a_m, u, a_{1},\! ..., a_i; b_n)\!\cdot\!
    E_{j-i,1}( a_{i+1}, \!... , a_j\,; v)
\otimes\\
\hspace{2.7in}
 \mu_E \!\left([\bar a_{j+1}|\dotsb
|\bar a_k]\otimes [\bar b_1|\dotsb |\bar
 b_{n-1}]\right),\\
\begin{array}{rll}
&\varepsilon _1 =\epsilon^a_p+(\epsilon^a_{p}+\epsilon^a_m)|v|,\\
&\varepsilon _2 = \epsilon^a_{m}+ (|u|+\epsilon^a_{k})(\epsilon^a_{k}+\epsilon^a_{m}) +
 (|v|+\epsilon^{b}_{n-1})(|b_n|+1)+  (\epsilon^a_{j}+\epsilon^a_{k}) (|v|+1),
\end{array}\\
 \text{where}\  \epsilon^x_{r}=|x_{1}|+\cdots + |x_{r}|+r.
\end{multline}

\begin{remark}
The first summand component of (\ref{formula}) agrees with (13) in \cite{KScubi} up to
signs: T he sign component $\epsilon^a_p$  correctly shown above  is omitted in
\cite{KScubi}.
\end{remark}
Thus, formula (\ref{formula}) gives the product on $\Lambda C^*(Y)$ by setting
$X=\operatorname{Sing}^1Y.$

\noindent For example, for $m=n=1$ we have (up to signs):
\begin{multline}\label{simpleproduct}
\lambda_E\left(\left( u\otimes [\,\bar a\,] \,\right)\otimes \left(v\otimes [\,\bar
b\,]\right)\right)= u\!\cdot\!v\otimes \left([\,\bar a\,|\,\bar b\,] +
 [\,\bar b\,|\,\bar a\,]+  E_{1,1}(a; b)\right)+\\
\hspace{2.0in} u\!\cdot\!E_{1,1}(a; v)\otimes [\,\bar b\,]+
  E_{1,1}(u;b)\!\cdot\!v\otimes [\bar a]+\\
  \left(E_{1,1}(u; b)\!\cdot\!E_{1,1}(a; v)+
   E_{2,1}(u, a; b)\!\cdot\!v+
  E_{2,1}(a, u;b)\!\cdot\!v \right)\otimes [\ ].
\end{multline}
Note that \cite{KScubi} $E_{1,1}$ is in fact Steenrod's original definition of the
 cochain $\smile_1$
operation. It satisfies the following Hirsch formula
$$  c \smile_1 a\!\cdot\!b=(c\smile _1a)\!\cdot\!b+(-1)^{|a|(|c|+1)}a\!\cdot\!(c\smile_1 b) $$
saying that $\smile_1$ is the left derivation with respect to the $\cdot$ (cup) product
on $C^*(Y).$ On the other hand, the map $-\!\!\smile _{1}\! c:C^*(Y)\rightarrow C^*(Y)$
is a derivation only {\it up to homotopy} with the operation $E_{2,1}$ serving as a
suitable homotopy:
\begin{multline*}
 dE_{2,1}(a,b\,;c)=
E_{2,1}(da,b\,;c)-(-1)^{|a|}E_{2,1}(a,db\,;c)+ (-1)^{|a|+|b|}E_{2,1}(a,b\,;dc)-\\
\hspace{1.1in}
 (-1)^{|a|}a\cdot b\smile_1\! c +(-1)^{|a|+|b|+ |b||c|}
(a\smile_1\!c)\cdot b + (-1)^{|a|}a\cdot(b\smile_1\! c),
\end{multline*}
the  \emph{Hirsch formula up to homotopy}. In the next subsection we point out the
other role of the operation $E_{2,1}.$
\subsection{Interaction between the Stasheff and Gerstenhaber higher order operations
on $\Lambda C^*(Y)$}\label{interaction}

Since the diagonal $\Delta_F$ is not coassociative,  the multiplication $\lambda_E$ on
the Hochschild chain complex $\Lambda C^*(Y)$ is not associative; but Theorem
\ref{main}(ii) and the acyclicity of $F_m\times I^n$ guarantees the existence of an
$A_{\infty}$-algebra structure on $\Lambda C^*(Y).$ Since  formula  (\ref{diagonal}),
it is expected to construct this structure precisely. Indeed, denoting $A=\Lambda
C^*(Y)$ and $x,y,z\in A$ with $x= u\otimes [\ ], y=v\otimes [\ ], z= 1\otimes [\,\bar
b\,], \, du=dv=db=0,$
  apply  (\ref{simpleproduct}) to obtain the following
  $\lambda_E$ products in $A$
  (up to sign):
  \[
  \begin{array}{lll}
  (xy)z=u\!\cdot \!v\otimes [\,\bar b\,]  +(u\!\cdot \!v\smile_1 v)\otimes [\ ]
  \ \ \ \text{and}\\
 x(yz)=u\!\cdot \!v\otimes [\,\bar b\,]  +
 \left((u\smile_1 b)\!\cdot \!v +u\!\cdot \!(v\smile_1 b)\right)
   \otimes [\ ].
\end{array}
   \]
  From
    the Hirsch formula up to homotopy immediately follows the equality
$$ d\left(E_{2,1}(u,v\,;b)\otimes [\ ] \right)= (xy)z-x(yz) .$$
Consequently, an operation $\varphi^3: A^{\otimes 3}\rightarrow A$ can be chosen with
$$\varphi^3\left(u\otimes [\ ]\,, \,v\otimes [\ ]\,,\,1\otimes [\,\bar b\,]\,\right)
=E_{2,1}(u,v\,;b)\otimes [\ ].$$

Note that the operation $E_{2,1}$ is unavoidable on the \emph{simplicial} cochain
algebra $C^*(Y)$ (i.e. there exists no  strict Hirsch formula for the both sides
simultaneously), and, consequently, its Hochschild chain algebra becomes a natural
occurring and simplest example in which there is a non-trivial $A_{\infty}$-algebra
structure on the chain level.

\section{Proof of Theorem 1}

Recall the definition of an hga $(A,d,\cdot\,,\{E_{p,q}\})_{p\geq 0,\,q=0,1},$ in
general \cite{Gerst-Voron}, \cite{Getz-Jones},  \cite{KScubi}. Given $k\geq 1,$ we have
the following defining identities for it:
\begin{equation}\label{dif}
\begin{array}{llll}
dE_{k,1}(a_1,...,a_k;b)&= & \sum _{i=1}^{k}\ \,\, (-1)^{\epsilon^a_{i-\!1}}
  \, E_{k,1}(a_1,..., da_i,...,a_k;b)\\  &  & \ \ \ \ +\ \
\,(-1)^{\epsilon^a_{k}}\ \ \  E_{k,1}(a_1,...,a_k;db)\\
 &  &\!\!\!\!\!\! +\sum_{i=1}^{k-1} \,
\,\,(-1)^{\epsilon^a_{i}}\ \ \ E_{k-1,1}(a_1,...,a_{i}a_{i+1},...,a_k;b) \\ & & \ \ \ \
+
 \ \ \,
 (-1)^{\epsilon^a_k+|a_k|\!|b|} \, E_{k-1,1}(a_1,...,a_{k-1};b)\!\cdot\! a_k\\ & & \ \ \ \ +  \
\ \, (-1)^{|a_1|}\ \ \ \ \ \ \   a_1\!\cdot\! E_{k-1,1}(a_2,...,a_k;b),
  \end{array}
\end{equation}
\begin{multline}\label{prod}
 E_{k,1}(a_1,\!...,a_k;b\cdot c)\\=
\sum_{i=0}^{k}(-1)^{|b| (\epsilon^a _{i}+\epsilon^a
_{k})}E_{i,1}(a_1,\!...,a_{i};b)\cdot E_{k-i,1}(a_{i+1},\!...,a_k;c)
\end{multline}
and
\begin{multline}\label{associativity}
 \sum _{\substack{_{k_1+\cdots+k_p=k}
 \\ _{1\leq p\leq k+\ell}}}
 (-1)^{\epsilon}
 E_{p,1}
\left(E_{k_1,\ell_1}(a_1,...,a_{k_1};b'_1),\!
...,E_{k_p,\ell_p}(a_{_{k-k_{p}+1}},...,a_k;b'_{p} )\,; c \right)
\\=
 E_{k,1}\left(a_1,...,a_k;E_{{\ell},1}(b_1,...,b_{\ell};c)\right),\\
b'_i\in \{1,b_1,..,b_{\ell}\},\ \  \
\epsilon=\sum_{i=1}^{p}(|b'_i|+1)(\varepsilon^a_{k_i}+\varepsilon^a_{k}),b'_i\neq 1.
\end{multline}

A \emph{morphism} $f:A\rightarrow A'$ between two hga's is a dga map $f$ commuting with
all $E_{p,q}.$ Obviously, formula (\ref{formula}) has a sense for the Hochschild chain
complex of an arbitrary  hga and then $f$ induces a dga map $\Lambda f: \Lambda
A\rightarrow \Lambda A', $ and, consequently, an algebra map  $ HH(f): HH(
A)\rightarrow HH( A'). $ However, in the  lemma below $f$ is not necessarily an  hga
map, nevertheless it induces an algebra map on the Hochschild homologies.

We  need the following lemma in which  for simplicity  the subscripts are removed for
the operations $E_{k,1}.$

\begin{lemma}\label{ederivation}
Let $A,A'$ be two $\Bbbk$-free  hga's and let  $f:A\rightarrow A'$ be a dga map such
that there is a sequence of maps $s=\{s_{k,1}:A^{\otimes k+1}\rightarrow A'\}_{k\geq 1}
$ with
$$fE_{k,1}-E'_{k,1}f^{\otimes k+1}=-(s_{k,1}+s_{k-1,1})D-d's_{k,1},\ \  \
  D:A^{\otimes k+1}\rightarrow
A^{\otimes k+1}\oplus A^{\otimes k},$$
 where $D(a_1,\!...,a_k;b)=\sum _{i=1}^{k}
(-1)^{\epsilon^a_{i-1}}
   (a_1,\!..., da_i,\!...,a_k;b)+
(-1)^{\epsilon^a_{k}} (a_1,\!...,a_k;db)
  + \sum_{i=1}^{k-1}
(-1)^{\epsilon^a_{i}} (a_1,\!...,a_{i}a_{i+1},\!...,a_k;b) ,$
 and
\begin{multline}\label{decomp}
 s(a_1,...,a_k;b\!\cdot\!c)\!=\!
  \sum_{i=0}^{n-1}
  (-1)^{(|b|\!+\!1) \epsilon^a _{i}+|b|(\epsilon^a_{k}\!+\!1)}
  E'(fa_1,...fa_i;fb)\!\cdot\!s(a_{i+1},...,a_k;c)\\
   +\sum_{j=1}^{n}  (-1)^{|b| (\epsilon^a _{j}+\epsilon^a
_{k})} s(a_1,...a_j;b)\!\cdot\!fE(a_{j+1},...,a_k;c),
\end{multline}
 and let $\mathbf{s}:BA\otimes BA\rightarrow BA'$ be the extension
 of a map $\bar{s}:BA\otimes BA\rightarrow A',$
$$
\bar {s}( [\bar a_1|\dotsb |\bar a_n]\otimes  [\bar b_1|\dotsm |\bar b_m])=\left \{
\begin{array}{lll}
s(a_1\otimes \cdots \otimes a_n\otimes b_1), & m=1\\
0, & \text{otherwise},
\end{array}
\right.
 $$
as a $(\mu_{E'}\circ(Bf\otimes Bf)\, ,Bf\circ \mu_E)$-coderivation. Then a map
$$\chi:\Lambda  A\otimes \Lambda A\rightarrow \Lambda A'$$ defined for
 $(u\otimes [\bar a_1|\dotsb |\bar a_n])\otimes (v\otimes [\bar
b_1|\dotsm |\bar b_m])\in \Lambda  A\otimes \Lambda A$ by
\begin{multline*}
\chi \left(\left(u\otimes [\bar a_1|\dotsm |\bar a_m]\right)\otimes
\left(v\otimes [\bar b_1|\dotsm |\bar b_n]\,\right)\right) \\
 =\sum _{p=0}^{m}\,
 (-1)^{\nu_1+|u|+\epsilon^a_{p}+|v|}\,
u\!\cdot\!fE( a_1,\!... , a_p; v) \otimes \mathbf{s}\left( [\bar a_{p+1
}|\dotsb |\bar a_m] \otimes [\bar b_1|\dotsb |\bar b_n]\right)\\
\hspace{-2.3in}+\sum _{p=0}^{m}\,
 (-1)^{\nu_1+|u|}\,
u\cdot\!s( a_1,\!... , a_p; v) \\
\hspace{1.5in}\otimes \left(\mu_{E'}\circ (Bf\otimes Bf)\right)\!\left( [\bar a_{p+1
}|\dotsb |\bar a_m] \otimes [\bar b_1|\dotsb |\bar b_n]\right)\\
\\
 \hspace{-0.2in}+
  \sum_{\substack{0\leq i\leq j\\ \leq k\leq m}}
  (-1)^{\nu_2+|u|+|b_n|+|v|+\epsilon^a_j+\epsilon^a_k+\epsilon^a_m}
 fE
 ( a_{k+1}, \!... , a_m, u, a_{1},\! ..., a_i; b_n)\\
\hspace{1.3in} \cdot
    fE( a_{i+1}, \!... , a_j ; v)
 \otimes\, \mathbf{s}\left([\bar a_{j+1}|\dotsb
|\bar a_k]\otimes [\bar b_1|\dotsb |\bar b_{n-1}]\right)\\
 \hspace{-0.35in}
 +
 \sum_{\substack{0\leq i\leq j\\ \leq k\leq m}}
   (-1)^{\nu_2+|u|+|b_n|+\epsilon^a_i+
 \epsilon^a_k+\epsilon^a_m}
 fE
 ( a_{k+1}, \!... , a_m, u, a_{1},\! ..., a_i; b_n)
 \\
 \ \ \ \ \ \ \ \
\cdot    s( a_{i+1}, \!... , a_j ; v)
 \otimes\, (\mu_{E'}\circ(Bf\otimes Bf)) \!\left([\bar a_{j+1}|\dotsb
|\bar a_k]\otimes [\bar b_1|\dotsb |\bar b_{n-1}]\right)\\
 \hspace{-0.3in}+
 \sum_{\substack{0\leq i\leq j\\ \leq k\leq m}}
  (-1)^{\nu_2}\,
 s( a_{k+1}
, \!... , a_m, u, a_{1},\! ..., a_i; b_n)\cdot\!
    E'(f a_{i+1}, \!... , fa_j ; fv)\\
\hspace{1.55in}
 \otimes\, (\mu_{E'}\circ(Bf\otimes Bf)) \!
\left([\bar a_{j+1}|\dotsb|\bar a_k]\otimes [\bar b_1|\dotsb |\bar b_{n-1}]\right)
\end{multline*}
is a chain homotopy between $\Lambda f\circ \lambda_E$ and
 $\lambda_{E'} \circ(\Lambda f\otimes \Lambda f).$

\end{lemma}
\begin{proof}
The proof is straightforward using  equality (\ref{decomp}). %\smartqed %\qed
\end{proof}
\begin{remark}
This lemma emphasizes a role of the explicit formula for the product $\lambda_E$ in the
following way: Since it has
 a general form $\sum (A_1\cdot
A_2\otimes A_3),$ the chain homotopy  $\chi$ admits to be of  the form $\chi=A_1\cdot
A_2\otimes s_3+A_1\cdot s_2\otimes A_3'+s_1\cdot A_2'\otimes A_3',$ i.e. a standard
derivation extension  of maps
 $s_i$  which are thought to be  related with the factors by
$A_i-A^{\prime}_i=ds_i+s_id.$

\end{remark}

Using  Lemma 1 we have the following comparison proposition.

\begin{proposition}\label{comparison}
Let $f:A\rightarrow A'$ be as in Lemma \ref{ederivation}. Then $f$ induces an algebra
map
\[HH(f): HH(A)\rightarrow HH(A')\]
and  when $H(f):H(A)\rightarrow H(A') $ is an isomorphism, so is $HH(f).$
\end{proposition}

Now let fix on $H$ the trivial hga structure, i.e. $\{E_{p,q}\}=\{E_{0,1},E_{1,0}\},$
while on $C^{\ast}(Y;\Bbbk)$ the canonical hga structure $\{E_{p,q}\}$  mentioned in
the previous section. Then we construct an auxiliary hga $(RH,d)$ with dg  algebra maps
$$H\overset{\rho}\longleftarrow RH\overset{f}
\longrightarrow C^{\ast}(Y;\Bbbk)$$ such that both maps satisfy the hypotheses of Lemma
1 and are cohomology isomorphisms. Then we can apply Proposition \ref{comparison} to
obtain algebra isomorphisms
 $$HH_*(H )
 \overset{_{HH(\rho)}}\longleftarrow HH_*(RH)
 \overset{_{HH(f)}}\longrightarrow HH_*(C^{\ast}(Y;\Bbbk)).$$
  Since the product
$\lambda_E$  on $\Lambda(H,0,\cdot,\{E_{0,1},E_{1,0}\})$ coincides with the standard
shuffle product,
 one gets  algebra isomorphisms \cite{Loday1}
$$HH_*(H,0,\cdot,\{E_{0,1},E_{1,0}\})\approx S(U)\otimes \Lambda( s^{_{-\!1}}\! U)
\approx H(Y;\Bbbk)\otimes H(\Omega Y;\Bbbk);$$  on the other hand,
 from the previous section  we have an
algebra isomorphism $$HH_*(C^{\ast}(Y;\Bbbk),d,\cdot \,,\{E_{p,q}\})\approx H^*(\Lambda
Y;\Bbbk).$$ Consequently, Theorem \ref{diff.forms} follows.

 Thus it remains to define the  hga $RH$ and maps $\rho,f$ mentioned
above: Indeed, consider a bigraded multiplicative resolution  $\rho :(RH,d)\rightarrow
H$ of $H$ (\cite{sane}, \cite{sane2}) such that $R^*H^*=T(V^{*,*})$ with
$V^{*,*}=V^{0,\ast}\oplus \mathcal{E}^{<0,*}\oplus \mathcal{T}^{-2r,*},r\geq 1,$\,
$V^{0,*}\approx U^*,$
 $\mathcal{E}^{-n,*}=
 \{\mathcal{E}_{k,1}^{-n,*} \}_{1\leq k\leq n}$ with
$\mathcal{E}_{k,1}^{-n,*}$ spanned on the set of
 expressions
$E_{k,1}(a_1,...,a_k;b),$ $  a_r\in R^{-i_r}H^*,$ $b\in V^{-j,*},$  $n=\sum_{r=1}^{k}
i_r+ j, $
 unless  $E_{1,1}(a;a),\,a\in V^{0,\ast},$ and
subjected   to relations  (\ref{associativity}),
 while $\mathcal{T}^{-2(n-1),*},\, n\geq 2,$ is spanned on
the set of expressions
  $a_1{\smallsmile}\!\!\!_{_{2}}\,a_2 {\smallsmile}\!\!\!_{_{2}}
  \cdots {\smallsmile}\!\!\!_{_{2}} \,a_n$ with $ a_i\in V^{0,*},\,
 a_i{\smallsmile}\!\!\!_{_{2}}\,a_j =
 a_j{\smallsmile}\!\!\!_{_{2}}\,a_i,$
   and
 $a_i\neq a_j$ for $i\neq j;$    the differential $d$
 is defined: On $V^{0,*}$  by $dV^{0,*}=0;$ on
$\mathcal{E}$ by formula (\ref{dif}), and  on $\mathcal{T}$  by
\begin{equation}\label{cup2q}
d (a_1{\smallsmile}\!\!\!_{_{2}}\cdots{\smallsmile}\!\!\!_{_{2}}a_n) =
\sum_{(\mathbf{i};\mathbf{j})}\,
(a_{i_1}{\smallsmile}\!\!\!_{_{2}}\cdots{\smallsmile}\!\!\!_{_{2}}\,a_{i_k})
\,{\smallsmile}\!\!\!_{_{1}}\,
(a_{j_1}{\smallsmile}\!\!\!_{_{2}}\cdots{\smallsmile}\!\!\!_{_{2}}\,a_{j_{\ell}})
\end{equation}
where the summation is over unshuffles $(\mathbf{i};\mathbf{j})=(i_1<\cdots <i_k\,
;j_1<\cdots<j_{\ell} )$ of $\underline{n}$ and ${\smallsmile}\!\!\!_{_{1}}$ denotes
$E_{1,1}.$
 In particular, $dE_{1,1}(a;b)=dE_{1,1}(b;a)=ab-ba$ and
$d(a{\smallsmile}\!\!\!_{_{2}}\,b)  =E_{1,1}(a;b)+E_{1,1}(b;a)$
 for $a,b\in V^{0,*}.$ It is
straightforward to check that $H(R^{<0}H^*,d)=0$ (or see the argument for an analogous
resolution, denoted by $R_{\delta}H$, in \cite{sane2}). Set $E_{1,1}(a; a)=0$ for $a\in
V^{0,\ast}$ and extend the operations $E_{k,1}(a_1,..,a_k; b),\, k\geq 1,$ by formula
(\ref{prod}) on $RH$ for any $b\in RH.$ Thus, $\rho$ becomes   an hga map too.
Consequently, $HH(\rho)$ is an isomorphism by Proposition \ref{comparison}.

Since $Sq_1 [z]=0,$ denote by $\gamma(z;z)\in C^{2n-2}(Y;\Bbbk)$
  a cochain such that
\begin{equation}\label{boundary}
d\gamma(z;z)=z\smile_1 z.
\end{equation}

Next define a dga map $f:(RH,d)\rightarrow C^{\ast}(Y;\Bbbk)$ as follows. First define
it on $V$ and then extend multiplicatively. On $V^{0,\ast}:$ by choosing cocycles
$f^0:U^*\rightarrow  C^*(Y;\Bbbk); $ on $\mathcal{T}$: by
 $f(a_1{\smallsmile}\!\!\!_{_{2}}\,a_2)=f^0a_1\smile_2 f^0a_2,$ where
  $\smile_2 $ denotes   Steenrod's cochain operation, and extend inductively
  for $a_1{\smallsmile}\!\!\!_{_{2}}\cdots {\smallsmile}\!\!\!_{_{2}}
   \,a_n,\ n\geq 3;$ such an extension has no obstructions, since a cocycle
   in $C^{*}(Y)$
   written by cochain operations
  in distinct variables is cohomologous to zero;
 on $\mathcal{E}:$
 for   $E_{k,1}(a_1,\!...,a_k;b),k\geq 1,$ with $a_i\in RH$ and  $b\in
V^{0,\ast}$ set
$$
fE_{k,1}(a_1,\!...,a_k;b)\!=\!\left\{
\begin{array}{llll}

E_{2,1}(fa_1,fa_2;fb)\\
-fa_1\!\cdot\! \gamma(fa_2;fb)-\gamma(fa_1;fb)\!\cdot\! fa_2,&&
                                               k=2
                                                \vspace{1mm}\\

E_{k,1}(fa_1, \!...,fa_k;fb), & &   k\neq 2;

\end{array}
\right.
$$
for $b=E_{\ell,1}(b_1,\!...,b_{\ell};c)$ and $ b'=E_{\ell-1,1}(b_1,\!...,b_{\ell-1};c)
,\, b''=E_{\ell-1,1}(b_2,\!...,b_{\ell};c) $  set
$$
fE_{k,1}(a_1,\!...,a_k;b)=\left\{
\begin{array}{llll}
E_{k,1}(fa_1, \!...,fa_k;fb)\\
+(-1)^{k}
 E_{k-1,1}(fa_1,\!...,fa_{k-1};fb_1)\!\cdot\!  \gamma(fa_k;fc)\\
 +(-1)^{k}
E_{k-1,1}(fa_1, \!...,fa_{k-1};fc)\!\cdot\!  \gamma(fa_k;fb_{1})\\ -

\gamma(fa_1;fb_1)\!\cdot\!  E_{k-1,1}(fa_2, \!...,fa_{k};fc)\\

-\gamma(fa_1;fc)\!\cdot\!  E_{k-1,1}(fa_2, \!...,fa_{k};fb_1),

& & \ell=1

 \vspace{0.1in}\\

E_{k,1}(fa_1, \!...,fa_k;fb) \\
+(-1)^{k} E_{k-1,1}(fa_1, \!...,fa_{k-1}; fb')\!\cdot\!
\gamma(fa_k;fb_{\ell})\\
- \gamma(fa_1;fb_{1})\!\cdot\!  fE_{k-1,1}(a_2, \!...,a_{k}; b''),
 & & \ell\geq 2,

\end{array}
\right.
$$
where we assume $\gamma(a;b)=0$ unless  $a= b$ with $a\in V ^{0,\ast}$ in which case
$\gamma$ is defined by (\ref{boundary});
 and, finally, for $b\in \mathcal{T},$
set
$$fE_{k,1}(a_1,...,a_k;b)=
  E_{k,1}(fa_1,...,fa_k;fb).$$

Define  maps $s_{k,1}: RH^{\otimes k+1}\rightarrow C^*(Y;\Bbbk),\,k\geq 1, $ of degree
$-1$ first for $a_1\otimes \cdots \otimes a_k \otimes b\in RH^{\otimes k}\otimes V$ by
$$
 s_{k,1}(a_1\otimes\cdots \otimes a_k\otimes b)=\left\{
 \begin{array}{lllll}

\gamma(b;b), && k=1, \, a_1=b \in V^{0,\ast}    \\

   0, & & \text{otherwise}

\end{array}
\right.
$$
and  then extend them on whole $RH^{\otimes k+1}$ by formula (\ref{decomp}). It is
immediate to verify that  $f$ and $s=\{s_{k,1}\}_{k\geq 1}$ satisfy the hypotheses of
Proposition \ref{comparison} and, consequently, $HH(f)$ is an isomorphism. \qed
%\smartqed

\begin{example}
Let $A=\Bbbk[x,y] ,$  $|x|=|y|=2,\,$   $B=T(\bar x,\bar y,z)/\{\bar{x}^2,\bar{y}^2 \},$
$|\bar x|=|\bar y|=|z|=1,$
 $ dz=\bar x\bar
y+\bar y\bar x.$ Take $(A,0)\otimes (B,d)$ and set $h(z)=x$ to  obtain the dga
$C=(A\otimes B,d^{\otimes}+h).$ It is easy to see that the spectral sequence of $C$ is
collapsed, and, consequently, its $E_{\infty}$-term is isomorphic as algebras with the
$E_{\infty}$-term of the Serre spectral sequence of the free loop fibration with the
base $Y=\Bbb{C}P^{\infty}\times \Bbb{C}P^{\infty}.$
 However,  $H^*(C)$ is isomorphic only additively with
$H^*(\Lambda Y).$

\end{example}

%\vspace{-0.5mm}

\end{document}